\documentclass[11pt,reqno]{amsart}
\usepackage{graphicx} 
\usepackage{amsmath,amssymb,amsthm}
\usepackage{url}
\usepackage{color}
\usepackage{tikz,tikz-cd}
\usepackage{dsfont}
\usepackage{comment}
\usepackage{cleveref}
\usepackage{enumitem}
\usepackage{thmtools}
\usepackage{thm-restate}
\usepackage{mathrsfs}
\usepackage{appendix}

\calclayout

\newtheorem{theorem}{Theorem}[section]
\newtheorem{lemma}{Lemma}[section]
\newtheorem{proposition}{Proposition}[section]
\newtheorem{remark}{Remark}[section]
\newtheorem{corollary}{Corollary}[section]

\DeclareMathOperator{\curl}{curl}

\DeclareMathOperator{\flux}{flux}
\newcommand{\hodge}{\!\star\!}
\newcommand{\iprod}{\mathbin{\lrcorner}}

\title{Spectral geometry of the curl operator on smoothly bounded domains}
\author{Josef E. Greilhuber}
\author{Willi Kepplinger}
\date{May 2025}

\begin{document}
\begin{abstract}
    We show that the spectrum of the curl operator on a generic smoothly bounded domain in three-dimensional Euclidean space consists of simple eigenvalues. The main new ingredient in our proof is a formula for the variation of curl eigenvalues under a perturbation of the domain, reminiscent of Hadamard's formula for the variation of Laplace eigenvalues under Dirichlet boundary conditions. As another application of this variational formula, we simplify the derivation of a well-known necessary condition for a domain to minimize the first curl eigenvalue functional among domains of a given volume and derive similar necessary conditions for a domain extremizing higher eigenvalue functionals.
\end{abstract}
\maketitle
\section{Introduction}

This paper is concerned with the spectrum of the operator 
\begin{align*}
    \curl u = \nabla \times u,
\end{align*}
which acts on vector fields $u$ defined on smoothly bounded domains in $\mathbb R^3$, i.e., bounded domains $D \subseteq \mathbb R^3$ with the property that for each point $p \in \partial D$ there exists an open neighborhood $U \subseteq \mathbb R^3$ of $p$ and a function $\rho \in C^\infty(U)$ with $U \cap D = \rho^{-1}((-\infty,0))$ and $d\rho(p) \neq 0$. Under appropriately chosen boundary conditions, the curl operator is self-adjoint with compact resolvent, so its spectrum consists of a discrete set of real eigenvalues of finite multiplicity.

The corresponding eigenfields describe magnetic fields in plasmas in which no Lorentz force acts on the charged particles constituting the plasma and are thus often referred to as ``force-free magnetic fields'' in the plasma physics literature. They can be shown to be energy-minimal magnetic fields (within a fixed helicity class of magnetic fields) and have been proposed to be the natural end configurations of dissipative magnetic fluids in which the magnetic forces dominate (\cite{Woltjer1958},\cite{LaurenceAvellaneda1991}).\par

More generally, curl eigenfields lie at the intersection of a number of mathematical subdisciplines. On the one hand, they are a particularly rich example class of solutions to the Euler equations, and, on the other hand, their orthogonal complements define contact structures \cite[Theorem 2.1]{Etnyre2000I}. This confluence is interesting from both an analytic as well as a contact geometric and topological point of view and has led to some remarkable flexibility results for the study of the Euler equations (e.g. \cite{Etnyre2000I},\cite{Etnyre2000III},\cite{Cardona2021}) and rigidity results in contact topology (e.g. \cite{Etnyre2012}).\par 

While there do exist studies of the spectral theory of the curl operator on domains in specific settings (e.g. \cite{CantarellaDeTurckGluckTeytel2000B},\cite{EncisoPeralta2020}), there have been few systematic treatments of the curl operator as a self-adjoint operator on domains.\par

\subsection{The setup}
 By imposing boundary conditions on the curl operator, we may obtain a Fredholm operator or an unbounded self-adjoint operator. One especially natural choice, introduced by Giga and Yoshida in \cite{GigaYoshida1990}, can be described as follows: Let $D\subset \mathbb R^3$ be a smoothly bounded domain and $\vec \nu$ its outward pointing unit normal. Define 
\begin{align*}
    L^2_\partial(D) = \left\{\vec u \in L^2(D,\mathbb C^3): \nabla \cdot \vec u = 0, \vec u \cdot \vec \nu = 0 \text{ on } \partial D\right\},
\end{align*}
the space of boundary-parallel, divergence free, and square integrable vector fields, where the divergence is defined in the weak sense. (It adds no difficulty and is helpful in the spectral theory to work over $\mathbb C$, which will be the convention for the rest of this paper.) Let
\begin{align*}
    H^1_{\partial\partial}(D) = \left\{\vec u \in L^2_\partial(D): \curl \vec u \in L^2_\partial(D)\right\},
\end{align*}
where $\curl$ is understood in the sense of distributions. It is shown in \cite{GigaYoshida1990} that $\curl$ is a Fredholm operator from $H^1_{\partial\partial}(D)$ to $L^2_\partial(D)$, whose spectrum is however all of $\mathbb C$ unless $\partial D$ has trivial first homology, i.e., $\partial D$ consists of a union of spheres. A more meaningful spectrum is obtained by introducing the spaces
\begin{align*}
    L^2_\Sigma(D) &= \left\{\vec u \in L^2_\partial(D): \textstyle\int_S \vec u \cdot \vec \nu_S = 0 \right\}, \\
    H^1_{\Sigma\Sigma}(D) &= \left\{\vec u \in L^2_\Sigma: \curl \vec u \in L^2_\Sigma(D)\right\}.
\end{align*}
where $S$ ranges over all compact embedded orientable surfaces in $D$ with $\partial S \subset \partial D$, and, given $S$, $\vec\nu_S$ is a unit normal vector field of $S$.
The main result of \cite{GigaYoshida1990} is that $\curl$ is an unbounded self-adjoint operator with compact resolvent on $L^2_\Sigma(D)$ with domain $H^1_{\Sigma\Sigma}(D)$.

There are other ways to restrict the space $H^1_{\partial\partial}(D)$ in order to obtain self-adjoint operators with compact resolvent. It follows from the work of Hiptmair, Kotiuga, and Tordeux \cite{HiptmairKotiugaTordeux2012} that these are indexed by Lagrangian subspaces of the first cohomology of the domain's boundary, denoted by $H^1_{dR}(\partial D,\mathbb C)$. The symplectic structure on $H^1_{dR}(\partial D, \mathbb C)$ is induced by the wedge product on $1$-forms, extended as a sesquilinear, antisymmetric form from $H^1_{dR}(\partial D,\mathbb R)$. We will discuss these operators, from which Giga--Yoshida's curl operator arises as a special case, in detail in \Cref{subsec:self-adjoint-realizations}. Given a Lagrangian subspace $L\subset H^1_{dR}(\partial D,\mathbb C)$ we will denote the associated curl operator by $\curl_L$. We wish to point out that Lagrangian boundary conditions encompass not just Giga--Yoshida's zero flux boundary conditions but also the ``Amperian'' boundary conditions proposed by Cantarella \cite{Cantarella1999}.

\subsection{Statement of results}
In this paper, we answer the question whether the spectrum of the curl operator with Lagrangian boundary conditions is generically simple. To state our result precisely, we formalize the ``space of domains'' as follows.

Let $D_0 \subseteq \mathbb R^3$ be a smoothly bounded domain. We consider the space $\mathcal X(D_0)$ of smooth embeddings of $\bar D_0$ into $\mathbb R^3$. As an open subset of the space of smooth maps from $\bar D_0$ to $\mathbb R^3$ this space naturally inherits the structure of a smooth Frech\'et manifold. Given a diffeomorphism $\Phi \in \mathcal X(D_0)$, we write $D_\Phi$ for $\Phi(D_0)$. Given a Lagrangian $L \subseteq H^1_{dR}(\partial D_0,\mathbb C)$, we write $L_\Phi$ for the image of $L$ under $(\Phi^{-1})^\ast: H^1_{dR}(\partial D_0,\mathbb C) \to H^1_{dR}(\partial D_{\Phi},\mathbb C)$. In the following theorem, we additionally require the Lagrangian $L \subseteq H^1_{dR}(\partial D_0,\mathbb C)$ to be \emph{real}, i.e.\ invariant under complex conjugation, which just means that $L$ arises from a Lagrangian in $H^1_{dR}(\partial D,\mathbb R)$ by complexification. This ensures that there exists a basis of real-valued eigenforms for $\curl_L$, an important element in our proof.

\begin{theorem}
    \label{thm:simple_eigenvalues}
    Let $D_0 \subseteq \mathbb R^3$ be a smoothly bounded domain and let $L\subset H^1_{dR}(\partial D_0,\mathbb R)$ be a Lagrangian subspace. For a comeagre subset of diffeomorphisms $\Phi \in \mathcal X(D_0)$ the spectrum of the operator $\curl_{L_\Phi}$ on $D_\Phi$ consists of simple eigenvalues.
\end{theorem}

Note that $L_\Phi$ is locally constant on the set of diffeomorphisms $\Phi \in \mathcal X(D_0)$ with a given image $\bar D \subseteq \mathbb R^3$, since any two isotopic diffeomorphisms $\Phi_1: \bar D_0 \to \bar D$ and $\Phi_2: \bar D_0 \to \bar D$ induce the same map on cohomology. Thus, the operator $\curl_{L_\Phi}$ and hence also its spectrum locally depend on $D$, but not on the choice of embedding $\Phi: \bar D_0 \to \bar D$. Therefore, \Cref{thm:simple_eigenvalues} really is a statement on generic \emph{domains} as opposed to generic embeddings $\Phi: \bar D_0 \to \mathbb R^3$.

We would like to point out that there have been no partial results in this direction, that is, this result has not been known for any boundary conditions or any domains.\par
The main ingredient in \Cref{thm:simple_eigenvalues} is a formula reminiscent of Hadamard's famous rule for the variation of eigenvalues of the Laplace operator with Dirichlet boundary conditions. If we deform the domain $D$ along the flow of a vector field $\vec X$, the first variation of an eigenvalue $\lambda$ of the $\curl_L$ operator is given by
\begin{align}
    \label{eq:hadamard_rule_introduction}
    \dot \lambda = - \lambda \int_{\partial D} (\vec X \cdot \vec \nu)\,|u|^2.
\end{align}
where $u$ is an eigenfield of $\lambda$. See \Cref{thm:hadamard_rule} for a precise statement.

The proof of Theorem \ref{thm:simple_eigenvalues} is informed by methods the authors have previously used in \cite{GreilhuberKepplinger2023} for spectral genericity results. It rests on the idea that sufficiently large families of operators with compact resolvent which do not satisfy generic spectral simplicity have eigenfunctions that satisfy strong pointwise constraints derived with the help of the Hadamard rule (\ref{eq:hadamard_rule_introduction}). These constraints are then shown to lead to a contradiction, thus establishing Theorem \ref{thm:simple_eigenvalues}.

To the best of our knowledge, formula \eqref{eq:hadamard_rule_introduction} is new even for the curl operator on domains diffeomorphic to the solid torus, where shape optimization problems for certain choices of Lagrangian boundary conditions have been considered previously in \cite{CantarellaDeTurckGluckTeytelA} and \cite{EncisoPeralta2020}, using variational formulas for the Rayleigh quotient outwardly similar to \eqref{eq:hadamard_rule_introduction}. Remarkably, the Hadamard rule for the curl operator bears a strong similarity to the corresponding formula, recently derived by Lamberti and Zaccharon \cite{LambertiZaccaron}, for the Maxwell system with various boundary conditions (see also Lamberti--Pauli--Zaccharon \cite{LambertiPauliZaccaron}). (Note that the boundary conditions of the Maxwell system, which is based on the differential operator $\curl^2$, preclude it from being the square of any self-adjoint curl operator \cite[Section 8]{HiptmairKotiugaTordeux2012}.)

A second application of the Hadamard rule is to significantly simplify certain calculations appearing in the treatment of optimal domains for the first curl eigenvalue, a topic which has recently attracted substantial interest (\cite{EncisoPeralta2020}, \cite{Gerner2023}, \cite{EncisoGernerPeralta2024}), as well as to generalize those results to arbitrary Lagrangian boundary conditions. We also deduce necessary conditions for a domain to be a local extremum of the $k^{th}$ eigenvalue functional.

As is customary, we fix the volume of $D$ to obtain a meaningful variational problem: Write $\mathcal X_1(D_0)$ for the subset of diffeomorphisms in $\mathcal X(D_0)$ with $|D_\Phi| = 1$. Let $L\in H^1_{dR} (\partial D_0;\mathbb C)$ be a Lagrangian. For $\Phi \in \mathcal X_1(D_0)$ and $k \in \mathbb N$, let $\lambda_k^L(\Phi)$ and $\lambda_{-k}^L(\Phi)$ denote the $k^{th}$ positive and negative eigenvalue of $\curl_{L_\Phi}$, respectively.

\begin{theorem}
    \label{thm:extremal_eigenvalues}
    Let $D_0 \subseteq \mathbb R^3$ and $L \subset H^1_{dR}(\partial D_0,\mathbb C)$.
     If $\Phi \in \mathcal X_1(D_0)$ is a local extremum of the $k^{th}$ eigenvalue functional $\lambda_k^L$, then there exists a family of pairwise orthogonal $\curl_{L_\Phi}$-eigenfields $(u_j)_{j=1}^\ell$ corresponding to the eigenvalue $\lambda_k^L(\Phi)$ which satisfy
    \begin{align*}
        |u_1|^2 + \ldots + |u_\ell|^2 = 1
    \end{align*}
    on the boundary of the domain $D_{\Phi}$. At local minima of $\lambda_1^L$ and local maxima of $\lambda_{-1}^L$, respectively, that eigenvalue is simple and the corresponding (suitably rescaled) eigenform $u$ satisfies $|u|^2 = 1$ on the boundary.
\end{theorem}

The last claim of \Cref{thm:extremal_eigenvalues} exists in the literature in certain special cases regarding the domain and boundary conditions (\cite{CantarellaDeTurckGluckTeytelA}, \cite{EncisoPeralta2020}, \cite{Gerner2023}). We provide it here again to show that it holds for any domain and Lagrangian boundary conditions. It should be noted that no domains which minimize $\lambda^L_1$ or maximize $\lambda^L_{-1}$ are known. In fact, it has been proven that there are no smooth axisymmetric domains (satisfying a certain additional geometric assumption on their boundary) which minimize $\lambda_1$ or maximize $\lambda_{-1}$ for Giga--Yoshida boundary conditions, see \cite{EncisoPeralta2020}. However, the existence of optimal domains within the class of convex domains was established in \cite{EncisoGernerPeralta2024}. (In this case, all Lagrangian boundary conditions coincide as the boundary topology is that of a sphere.) As far as the authors are aware, nothing is known about the existence of domains which extremize higher eigenvalues of the curl operator.

\subsection{Organization of this paper.} In Section \ref{section: preliminary observations} we review the self-adjoint realizations of the curl operator on smoothly bounded domains in the language of differential $2$-forms, which turns out to be advantageous for the perturbation theory employed later on. In particular, we relate the setting for the analysis of the curl operator of Giga and Yoshida \cite{GigaYoshida1990} to that of Hiptmair, Kotiuga, and Tordeux \cite{HiptmairKotiugaTordeux2012} in Subsection \ref{subsec:self-adjoint-realizations}. Section \ref{section: analytic dependence of the resolvent} is devoted to proving a multi-parameter analytic dependence lemma (Lemma \ref{lem:analytic_resolvent}) for $\curl_L$. This will be used both in the proof of the Hadamard rule for $\curl_L$ (Theorem \ref{thm:hadamard_rule}) in Section \ref{section: The Hadamard rule} as well as for that of Theorem \ref{thm:simple_eigenvalues} in Section \ref{section: generic simplicity of the spectrum}. Finally, Theorem \ref{thm:extremal_eigenvalues} is proven in Section \ref{section: critical points of the eigenvalue functionals}.

\subsection{Acknowledgements}
The first author is thankful to his advisor, Eugenia Malinnikova, for her expert guidance, many helpful discussions and insightful remarks. The second author thanks his advisor Vera V\'ertesi as well as Michael Eichmair for their constant support and helpful mentoring as well as the Vienna School of Mathematics for providing a stable and pleasant research environment. The authors express their heartfelt thanks to Thomas Körber for his meticulous reading of this manuscript and many valuable suggestions and to Davide Buoso for a friendly and informative discussion. This research was funded in part by the Austrian Science Fund (FWF) [10.55776/P34318] and [10.55776/Y963], as well as by the European Research Council (ERC) Consolidator Grant through the grant agreement 101001159. For open access purposes, the authors have applied a CC BY public copyright license to any author-accepted manuscript version arising from this submission.

Much of the relevant research was conducted during a research stay at the ICMAT in April 2024. The authors wish to express their gratitude to the institution and to Daniel Peralta-Salas in particular for their hospitality.

\section{Preliminary observations}\label{section: preliminary observations}

Although we are interested in the curl operator on domains in $\mathbb R^3$, much of the analysis of this situation is best expressed through the curl operator with respect to general Riemannian metrics on a fixed smoothly bounded domain. The curl operator arises from the exterior derivative on $1$-forms on three-dimensional smooth manifolds. With the Hodge-star operator induced by a Riemannian metric $g$, one obtains a degree-preserving operator, e.g.\ the operator $\star \circ d$ acting on $1$-forms as in \cite{HiptmairKotiugaTordeux2012}.
For variational arguments, we require a function space which is independent of the metric. This leads us instead to consider the operator $d \circ \star$ on $2$-forms: After identifying vector fields with $2$-forms via $u(\vec v,\vec w) := \det_g(\vec u, \vec v, \vec w)$, $d\circ \star$ corresponds to the curl operator on vector fields, and the conditions $\nabla \cdot u = 0$ and $\vec u \cdot \vec \nu = 0$ take the metric-independent and diffeomorphism-invariant shape $du = 0$ and $\iota_{\partial D}^\ast u = 0$. (Here and in the rest of this paper, $\iota_{\partial D}$ denotes the embedding $\partial D \hookrightarrow D$.)

\subsection{Various domains for the curl operator}\label{subsection: domains of curl}

Giga and Yoshida's curl operator, when expressed in the language of $2$-forms on a compact Riemannian three-manifold $(D,g)$ with boundary, is an unbounded self-adjoint operator on the space
\begin{align*}
    \begin{split}
    L^2_\Sigma(D) = \left\{u \in L^2(D,\Lambda^2 T^\ast D): du = 0, \iota_{\partial D}^\ast u = 0, \textstyle \int_S u = 0\right\}.
    \end{split}
\end{align*}
Here, $S$ ranges over all compact orientable surfaces in $D$ with boundary in $\partial D$. The exterior derivative in this definition must be understood in the weak sense. The functionals $u \mapsto \iota_{\partial D}^\ast u$ and $u \mapsto \textstyle \int_S u$ are bounded on the space of weakly closed $L^2$-forms, as follows from the Weyl decomposition \cite[Chapter 2]{Schwarz1995}. Hence, $L^2_\Sigma(D)$ is indeed a Banach space.
Note that $du = 0$ and $\iota_{\partial D}^\ast u = 0$ imply that $\int_S u$ depends only on the relative homology class of $S$ in $H_2 (\bar D,\partial D;\mathbb C)$. Thus, the condition $\int_S u = 0$ adds only a finite number of linear constraints.

The domain of Giga--Yoshida's $\curl$ operator is the space
\begin{align*}
    H^1_{\Sigma\Sigma}(D,g) = \left\{u \in L^2_\Sigma(D): d \hodge u \in L^2_\Sigma(D) \right\}.
\end{align*}
This space, unlike $L^2_\Sigma(D)$ itself, depends \emph{as a set} on the metric $g$, since the operator $\star$ occurs in its definition and involves the metric. We may write $H_{\Sigma\Sigma}^1(D,g)$ and $L^2_\Sigma(D)$ to emphasize this distinction and will keep to this convention in \Cref{section: preliminary observations} and \Cref{section: analytic dependence of the resolvent}. One should think of $L^2_\Sigma(D)$ as a Banach space equipped with a fixed reference norm equivalent to any norm arising from an inner product corresponding to a choice of smooth metric on $D$.

Dropping the \emph{zero-flux} condition $\int_S u = 0$ in the definition of $L^2_\Sigma (D)$ results in the space
\begin{align*}
    \begin{split}
    L^2_\partial(D) = \left\{u \in L^2(D,\Lambda^2 T^\ast D): du = 0, \iota_{\partial D}^\ast u = 0 \right\},
    \end{split}
\end{align*}
which is the direct sum of $L^2_\Sigma(D)$ and the finite-dimensional space of harmonic $2$-forms with vanishing tangential component, see \cite[Corollary 2.6.2]{Schwarz1995}. The space of harmonic $2$-forms will be denoted by $\mathcal H_\partial(D,g)$. Extending the domain of $\curl$ to 
\begin{align*}
    H^1_{\Sigma\partial}(D,g) = \left\{u \in L^2_\Sigma(D): d \hodge u \in L^2_\partial(D) \right\},
\end{align*}
we obtain a Fredholm operator with the same spectrum as before, which is, however, no longer self-adjoint \cite[Remark 2]{GigaYoshida1990}. Extending the domain further to 
\begin{align*}
    H^1_{\partial\partial}(D,g) = \left\{u \in L^2_\partial(D): d \hodge u \in L^2_\partial(D) \right\} = H^1_{\Sigma\partial}(D,g) \oplus \mathcal H_\partial(D,g),
\end{align*}
we obtain yet another Fredholm operator whose (point) spectrum, however, is all of $\mathbb C$, unless the space of harmonic forms $\mathcal H_\partial(D,g)$ is trivial \cite[Theorem 2]{GigaYoshida1990} which is true if and only if $\partial D$ is a union of spheres. (Our nomenclature for these spaces is essentially that of \cite{GigaYoshida1990}, except that we use the symbol $\partial$ instead of $\sigma$ to indicate tangential boundary conditions.)

To deal with variations in the metric, it is convenient to analyze the curl operator on a domain that does not depend on the metric $g$. This is achieved by considering the spaces 
\begin{align*}
    H^1_\partial(D) &= \{u \in H^1(D,\Omega^2(D)): du = 0, \iota_{\partial \Omega}^\ast u = 0\}, \\
    H^1_\Sigma(D) &= \{u \in H^1_\partial(D): \textstyle\int_S u = 0\} \text{ and } \\
    L^2_{cl}(D) &= \{u \in L^2(D,\Omega^2(D)): du = 0, \textstyle\int_{C} u = 0 \text{ for all connected components } C \subseteq \partial D\}.
\end{align*}

Note that the functional $u\mapsto\int_C u$ is bounded on the space of \emph{closed} $L^2$ forms (which is itself a closed subspace of $L^2(D,\Omega^2(D))$). Hence $L^2_{cl}(D)$ as defined above is a closed subspace of $L^2(D,\Omega^2(D))$. It follows e.g.\ from \cite[Theorem 1.1]{EncisoGarciaPeralta2018} that $\curl$ is a Fredholm operator from $H^1_\partial(D)$ to $L^2_{cl}(D)$, that its kernel is $\mathcal H_\partial(D,g)$, and that it is invertible from $H^1_\Sigma(D)$ to $L^2_{cl}(D)$. The relation of all spaces above is summarized by the following commutative diagram.

\begin{center}
\begin{tikzcd}
|[yshift=-0.6cm,xshift=-0.5cm,overlay]| \mathcal H_\partial(D,g) \arrow[r, hookrightarrow] \arrow[rd, hookrightarrow] & H^1_{\partial}(D) \arrow[r] & H^1_{\Sigma}(D) \arrow[r, "\curl_g"] & L^2_{cl}(D) \\
& H^1_{\partial\partial}(D) \arrow[r] \arrow[u,hookrightarrow] & H^1_{\Sigma\partial}(D) \arrow[r, "\curl_g"] \arrow[u,hookrightarrow] & L^2_{\partial}(D) \arrow[u,hookrightarrow] \\
& & H^1_{\Sigma\Sigma}(D) \arrow[r, "\curl_g"] \arrow[u,hookrightarrow]& L^2_{\Sigma}(D) \arrow[u,hookrightarrow]
\end{tikzcd}
\end{center}

For the convenience of the reader, let us collect the function spaces mentioned above with their definitions here once more.
\allowdisplaybreaks
\begin{align*}
     L^2_{cl}(D) &= \{u \in L^2(D,\Omega^2(D)): du = 0, \textstyle \int_{C} u = 0 \text{ for all connected components $C$ of $\partial D$}\} \\
     L^2_{\partial}(D) &= \{u \in L^2_{cl}(D):\iota_{\partial D}^\ast u = 0\} \\
     L^2_{\Sigma}(D) &= \{u \in L^2_{\partial}(D): \textstyle\int_S u = 0 \text{ for all $S \subset D$ with $\partial S \subset \partial D$}\} \\
     H^1_\partial(D) &= \{u \in H^1(D,\Omega^2(D)): du = 0, \iota_{\partial \Omega}^\ast u = 0 \} \\
     H^1_\Sigma(D) &= \{u \in H^1_\partial(D): \textstyle\int_S u = 0 \text{ for all $S \subset D$ with $\partial S \subset \partial D$}\} \\
     H^1_{\Sigma\partial}(D,g) &= \{u \in H^1_\Sigma(D): d\!\star\! u \in L^2_{\partial}(D)\} \\
     H^1_{\Sigma\Sigma}(D,g) &= \{u \in H^1_\Sigma(D): d\!\star\! u \in L^2_{\Sigma}(D)\} \\
     \mathcal H_\partial(D,g) &= \{u \in L^2_\partial(D): d\!\star\! u = 0\} \\
     H^1_{\partial\partial}(D,g) &= H^1_{\Sigma\partial}(D,g) \oplus \mathcal H_\partial(D,g)
\end{align*}

\subsection{Self-adjoint realizations of the curl operator}
\label{subsec:self-adjoint-realizations}

By restricting the domain of the operator $\curl: H^1_{\partial\partial}(D,g) \to L^2_\partial(D)$ to carefully chosen subspaces of $H^1_{\partial\partial}(D,g)$, a wealth of self-adjoint curl operators arises. These operators correspond to the self-adjoint extensions of $\curl$ based on closed traces, as presented in \cite{HiptmairKotiugaTordeux2012}, but restricted to the space of boundary-parallel, closed $2$-forms. Let us briefly recall the discussion in \cite{HiptmairKotiugaTordeux2012} as it applies in our setting.

We begin by considering the following identity, valid for all $u,v \in H^1_{\partial\partial}(D,g)$:
\begin{align}
    \label{eq:integration_by_parts}
    \left(\curl u, v\right)_{L^2_\partial(D,g)} - \left(u, \curl v\right)_{L^2_\partial(D,g)} = \int_{\partial D} \star u \wedge \star \bar v
\end{align}
Since $d\hodge u \in L^2_\partial(D)$, $\iota_{\partial D}^\ast(\star u)$ is a closed $1$-form on $\partial D$. Let $\ell$ denote the genus of $\partial D$. Then the de Rham cohomology group $H^1_{dR}(D,\mathbb C)$ is of rank $2\ell$, and the right hand side of \eqref{eq:integration_by_parts} descends to a sesquilinear symplectic form $\omega_D$ on $H^1_{dR}(D,\mathbb C)$.

Consider a Lagrangian subspace $L \subseteq H^1_{dR}(\partial D,\mathbb C)$, i.e.,\ an $\ell$-dimensional complex subspace such that $\omega_D(X,Y) = 0$ for all $X,Y \in L$. By Hiptmair--Kotiuga--Tordeux' work, restricting $\curl$ to the closed subspace
\begin{align*}
    H^1_L(D,g) := \{u \in H^1_{\partial\partial}(D,g): [\star u] \in L\},
\end{align*}
yields an unbounded self-adjoint operator on $L^2_\partial(D,g)$. In fact, their work shows all subspaces of $H^1_{\partial\partial}(D,g)$ to which $\curl$ restricts to an unbounded self-adjoint operator on $L^2_\partial(D)$ are of this form for some Lagrangian $L \subseteq H^1_{dR}(D,\mathbb C)$ (See \cite[Section 6]{HiptmairKotiugaTordeux2012}, especially \cite[Theorem 6.4]{HiptmairKotiugaTordeux2012}).

\subsubsection{Reinterpreting the zero-flux boundary condition}
\label{sec:zero_flux_reinterpretation}

The zero-flux boundary conditions can be seen as a special case of this setup as follows: Any surface $\Sigma \subset D$ with boundary on $\partial D$ gives rise to a cycle $\partial \Sigma$ in $\partial D$. The space $L_\Sigma \subseteq H^1_{dR}(\partial D; \mathbb C)$ cut out by all linear equations of the form $\int_{\partial \Sigma} u = 0$ arising from surfaces $\Sigma\subset D$ with boundary in $\partial D$ is Lagrangian. Indeed, using standard methods (see, for example, \cite{Kotiuga1987}) one can show that there exists a basis $[\Sigma_1],\dots,[\Sigma_{b_2 (D,\,\partial D)}]$ of the relative homology group $H_2 (\bar D,\partial D; \mathbb C)$ such that the homology classes $[\Sigma_i]$ are represented by compact embedded orientable surfaces $\Sigma_i$ in $D$ with boundary in $\partial D$. The image of the boundary map $\partial: H_2 (\bar D,\partial D; \mathbb C) \to H_1 (\partial D; \mathbb C)$ is a Lagrangian subspace (\cite[Proposition 9.1.4.]{Martelli2016AnIT}) with respect to the homological intersection number and so we may conclude that the collection of homology classes represented by $\partial \Sigma_1,\dots,\partial\Sigma_{b_2 (D,\,\partial D)} \subset \partial D$ generate a Lagrangian subspace $\Tilde{L}_\Sigma$ of $H_1 (\partial D, \mathbb C)$. As both the Poincar\'e duality isomorphism and the de Rham isomorphism preserve the intersection form \cite[Theorem 5.45.]{Warner1983}, the annihilator $L_\Sigma = \{[u] \in H^1_{dR}(\partial D, \mathbb C): \int_{\partial \Sigma} u = 0 \}$ of $\Tilde{L}_\Sigma$ is then a Lagrangian subspace in $H^1_{dR}(\partial D,\mathbb C)$.
It gives rise to the curl operator domain
\begin{align*}
    H^1_{L_\Sigma}(D,g) &= \left\{u \in H^1_{\partial \partial}(D,g): \textstyle\int_S \curl u = 0 \right\} = H^1_{\Sigma\Sigma}(D,g) \oplus \mathcal H_\partial(D,g).
\end{align*}
We conclude that the curl operator obtained from $L_\Sigma$ differs from Giga--Yoshida's curl operator only by the presence of its kernel, $\mathcal H_\partial(D,g)$, which is excised by the full zero-flux boundary conditions (which additionally require $\int_\Sigma u = 0$). While Giga--Yoshida's curl operator is an unbounded self-adjoint operator on $L^2_\Sigma(D,g)$, the curl operator obtained from $L_\Sigma$ is an unbounded self-adjoint operator on the larger space $L^2_\partial(D,g)$.

\section{Analytic dependence of the resolvent}\label{section: analytic dependence of the resolvent}

In this section, we show that resolvent of $\curl_L$ varies (real)-analytically along analytic perturbations of the domain and also depends analytically on the Lagrangian $L$. It is convenient to prove these statements in the more general setting of a smoothly bounded domain $D$ equipped with a metric $g_t$ which depends analytically on a parameter $t \in (-\varepsilon,\varepsilon)^k$. The case of a $k$-parameter variation of the domain boundary is encompassed by this: If $\Phi_t$ is an analytic $k$-parameter family of diffeomorphisms, then $g_t = \Phi_t^\ast g$ is an analytic $k$-parameter family of smooth metrics on $D$ and $\curl$ on $D_t = \Phi(D)$ is conjugate to $\curl_{g_t}$ on $D$ via the diffeomorphism $\Phi_t$. Furthermore, we also prove Lipschitz-continuous dependence of the resolvent with respect to $C^2$-smooth perturbations of the domain. 

Let us stress once more that the domain $H^1_{L}(D,g_t)$ of $\curl_{t,L}$ depends (as a set) on the metric $g_t$ and on the Lagrangian $L$, but the codomain, $L^2_{\partial}(D)$, is independent of both. Let $\iota: H^1_{L}(D,g_t) \to L^2_{\partial}(D)$ denote the canonical embedding, let $\sigma (\curl_{t,L})$ denote the spectrum of $\curl_{t,L}$ and let $\lambda \in \mathbb C$. The inverse of $(\lambda \iota - \curl_{t,L}): H^1_{L}(D,g_t) \to L^2_{\partial}(D)$ exists whenever $\lambda \not\in \sigma(\curl_{t,L})$. The operators $R_{t,L}(\lambda) = \iota \circ (\lambda \iota - \curl_{t,L})^{-1}$ then form a family of bounded operators on a \emph{fixed} space. This family is analytic in $t$ and $L$ (interpreting $L$ as a point in the Grassmanian $\mathrm{LGr}_\ell(H^1_{dR}(\partial D,\mathbb C))$ of complex Lagrangians in $H^1_{dR}(\partial D,\mathbb C)$, which is a real-analytic submanifold of the complex Grassmannian $\mathrm{Gr}_\ell(H^1_{dR}(\partial D,\mathbb C))$, as follows from \cite[Theorem 3]{EverittMarkus1999}).

\begin{lemma}
    \label{lem:analytic_resolvent}
    Let $D \subseteq \mathbb R^3$ be a smoothly bounded domain. Let $g_t$ be a family of smooth metrics on $\overline{D}$ depending real-analytically on $t \in (-\varepsilon,\varepsilon)^k$. To each $t \in (-\varepsilon,\varepsilon)^k$ and Lagrangian $L \in \mathrm{LGr}(H^1_{dR}(\partial D,\mathbb C))$, we associate the operator $\curl_{t,L}: H^1_{L}(D,g_t) \to L^2_\partial(D)$.
    
    Then the resolvents $R_{t,L}(\lambda): L^2_\partial(D) \to L^2_\partial(D)$ of $\curl_{t,L}: H^1_{L}(D,g_t) \to L^2_\partial(D)$ form a family of compact operators which is meromorphic in $\lambda$ for fixed $t$ and $L$ and real-analytic in $t$ and $L$ near any $(t,L,\lambda) \in (-\varepsilon,\varepsilon)^k \times \mathrm{LGr}(H^1_{dR}(\partial D,\mathbb C)) \times \mathbb C$ where $R_{t,L}(\lambda)$ is defined.
\end{lemma}

\begin{proof}
    The Hodge star operators $\star_t: H^1(\Omega^2(D)) \to H^1(\Omega^1(D))$ form a real-analytic family of bounded operators. It extends to a holomorphic family $\star_z$ of bounded operators on some neighborhood $U$ of $(-\varepsilon,\varepsilon)^k$ in $\mathbb C^k$. The operators $\curl_z := d \circ \star_z$ thus form a holomorphic family of bounded operators from $H^1_{\partial}(D)$ to $L^2_{cl}(D)$.
    
    Let us also choose a system of cut surfaces $\Sigma_1,\ldots,\Sigma_\ell \subset D$ which form a basis of the relative homology $H_2(\bar D, \partial D)$ (cf. \Cref{sec:zero_flux_reinterpretation}) and consider the map $\mathrm{flux}: H^1_\partial(D) \to \mathbb C^\ell$ given by
    \begin{align*}
        \mathrm{flux}(u) = \bigg(\int_{\Sigma_j} \! u\bigg)_{j=1}^\ell,
    \end{align*}
    which is bounded by the trace lemma. The restriction of smooth functions on $D$ to $\Sigma_j$ extends to a bounded linear functional from $H^1(D)$ to $L^{2}(\Sigma_j)$. Note that $\mathrm{flux}(u)$ is independent of $g_t$ and the choices of $\Sigma_j$ in $[\Sigma_j]$, the latter since $u$ is closed and $\iota_{\partial D}^\ast u = 0$. 

    Lastly, we introduce boundary conditions of the form
    \begin{align*}
        \sum_{j=1}^\ell \left(a_j \int_{\alpha_j} \! \star u + b_j \int_{\beta_j} \! \star u\right) = 0,
    \end{align*}
    where $a_j,b_j \in \mathbb C$ and $\alpha_j$, $\beta_j$ form a basis of $H_1(\partial D)$. (See \Cref{fig: alpha beta curve systems} for an illustration of the curve system $\alpha_j,\beta_k$.) A minor inconvenience which occurs here is that $H^1(\Omega^2(D))$ offers insufficient regularity to ensure continuity of the functional $u \mapsto \int_{\alpha_j} \star u$ since the curves $\alpha_j$ and $\beta_j$ are of codimension $2$. We fix this by considering a tubular neighborhood of each curve $\alpha_j$ in $\partial D$ which is foliated by curves $\{ \alpha_j^\tau, \tau \in [0,1]\}$ and introducing the functional $u \to \int_0^1\big(\int_{\alpha_j^\tau} \star u\big) d\tau$, which is bounded by the trace lemma. For $u \in H_{\partial\partial}(D)$, these two functionals agree as $\iota_{\partial D}^\ast \hodge u$ is closed. Let $\ell$ denote the genus of $\partial D$. Define
    \begin{align*}
        B&: H^1(\Omega^2(D)) \to \mathbb C^{2\ell} \\
        B(u) &= \left( \int_0^1\!\left(\int_{\alpha_j^\tau} \star u\right) d\tau , \int_0^1\!\left(\int_{\beta_j^\tau} \star u\right) d\tau \right)_{j=1}^{\ell}
    \end{align*}

     Let $f_1,\ldots,f_\ell \in (\mathbb C^{2\ell})^\ast$ be $\ell$ linearly independent linear functionals such that $[\star u]$ lies in the Lagrangian $L$ if and only if $f_j \circ B(u) = 0$ for $j=1,\ldots,\ell$. Write $B_L: H^2_{\partial}(D) \to \mathbb C^\ell$, $B_L = \left(f_j \circ B\right)_{j=1}^\ell$, so that $H^1_L(D) = H^1_{\partial\partial}(D) \cap B_L^{-1}(\{0\})$. On a sufficiently small chart of $\mathrm{LGr}_\ell(H^1_{dR}(\partial D,\mathbb C))$ we may choose $f_1,\ldots,f_\ell$ which depend real-analytically on $L$. With this choice, the operator $B_L$ then depends real-analytically on $L$. Let $V \subseteq \mathbb C^{\dim_{\mathbb R} \mathrm{LGr}_\ell(H^1_{dR}(\partial D,\mathbb C))}$ denote a neighborhood of the chart to which $B_L$ extends holomorphically.
    
    It follows e.g.\ from \cite[Theorem 1.1]{EncisoGarciaPeralta2018} that $\curl_t: H^1_{\Sigma}(D) \to L^2_{cl}(D)$ is invertible for real $t$. Since $H^1_{\partial}(D) = H^1_{\Sigma}(D) \oplus \mathcal H_{\partial}(D,g_t)$ and $\flux: \mathcal H_{\partial}(D,g_t) \to \mathbb C^{\ell}$ is bijective (\cite[Corollary 2.6.2]{Schwarz1995},\cite[Theorem 1]{CantarellaDeTurckGluck2002topology}), $\curl_t \oplus \flux$ has an inverse $A_t: L^2_{cl}(D) \oplus \mathbb C^{\ell} \to H^1_\partial(D)$. Let $\iota: H^1_{\partial}(D) \to L^2_{cl}(D)$ denote the natural embedding. The formula 
    \begin{align*}
        A_z&: L^2_{cl}(D) \to H^1_{\Sigma}(D) \\
        A_z &:= A_t(\mathds{1} - (\curl_z \oplus \flux - \curl_t \oplus \flux) \circ A_t)^{-1}
    \end{align*}
    provides an inverse to $\curl_z \oplus \flux$ for all $z \in U$ where $\mathds{1} - (\curl_z \oplus \flux - \curl_t \oplus \flux) \circ A_t$ is invertible. After restricting $U$ to a smaller neighborhood of $(-\varepsilon,\varepsilon)^k$, we may assume this is the case on all of $U$, since $\curl_z$ depends continuously on $z$ in the norm topology on $\mathcal L(H^1_\partial(D),L^2_{cl}(D))$. Hence, $(A_z)_{z \in U}$ is an analytic family of operators on $L^2_{cl}(D)$.
    Next, consider the family of operators
    \begin{align*}
        R(s,t,L,\lambda)&: L^2_{cl}(D) \oplus \mathbb C^{\ell} \to L^2_{cl}(D) \oplus \mathbb C^{\ell} \\
        R(s,t,L,\lambda) &= \mathds{1} - \big( \lambda \iota \oplus (s \flux - s B_L) \big) \circ A_t,
    \end{align*}
    which is analytic in $t \in U$, $L \in V$ and $s, \lambda \in \mathbb C$. Note that $\lambda \iota \oplus (s \flux - s B_L)$ is a compact operator from $H^1_\partial(D)$ to $L^2_{cl}(D) \oplus \mathbb C^\ell$. Since $R(0,t,L,0) = \mathds{1}$ is invertible, the hypotheses of a multivariable version of the Analytic Fredholm Theorem are met \cite[Theorem 3]{StessinYangZhu2011}. It follows that there exists an analytic variety $\mathcal S \subseteq \mathbb C \times U \times V \times \mathbb C$ such that $R(s,t,L,\lambda)$ is invertible if $(s,t,L,\lambda) \not\in \mathcal S$ and such that $R(s,t,L,\lambda)^{-1}$ is analytic on the complement of $\mathcal S$.
    
    Furthermore, the points $(s,t,L,\lambda) \in \mathcal S$ are precisely those where $R(s,t,L,\lambda)$ has a finite-dimensional kernel. In this case, there exist $(v,f) \in L^2_{cl}(D) \oplus \mathbb C^{\ell}$ such that $v = \lambda u$ and $f = s \flux u - s B_L(u)$, where $u = A_t (v,f)$ is the unique form satisfying $\flux u = f$ and $\curl_t u = v$. If $s=1$, the form $u$ is then an eigenform of $\curl_{t,L}$ in $H^1_L(D,g_t)$ corresponding to the eigenvalue $\lambda$.
    
    Next, we analyze the complement of $\mathcal S$, i.e.\ those points where $R(1,t,L,\lambda)$ is invertible. Consider $(v,f) \in L^2_{cl}(D) \oplus \mathbb C^{\ell}$ and write 
    \begin{align*}
        (v',f') &= R(1,t,L,\lambda)^{-1}(v,f), \
        A_t (v',f') = u'
    \end{align*}
    Then $R(1,t,L,\lambda) (v',f') = (v,f)$ implies
    \begin{align*}
        v' - \lambda u' &= \curl u' - \lambda u' = v, \\
        f' - f' + B_L(u') &= B_L(u') = f.
    \end{align*}
    Thus, $(\curl - \lambda \iota) \oplus B_L$ has an explicit inverse in 
    \begin{align*}
        A_t \circ R(1,t,L,\lambda)^{-1},
    \end{align*}
    whenever $\lambda$ is not contained in the spectrum of $\curl_L$. By construction, this operator is analytic in $t$, $L$ and $\lambda$. The resolvent of $\curl_{t,L}: H^1_L(D,g_t) \to L^2_{cl}(D)$, interpreted as a bounded linear operator on $L^2_{cl}(D)$, is given by restricting
    \begin{align*}
        R_{t,L}(\lambda) := \iota \circ A_t \circ R(1,t,L,\lambda)^{-1}
    \end{align*}
    to $L^2_\partial(D) \simeq L^2_\partial(D) \oplus \{0\} \subseteq L^2_{cl}(D)$. Thus, the resolvent $R_{t,L}(\lambda)$ is meromorphic in $\lambda$ and real-analytic in $t$ and $L$ near any $(t,L,\lambda)$ where it is defined.
\end{proof}

\begin{remark}
  In the proof of Lemma \ref{lem:analytic_resolvent} we used that the map $\flux:\mathcal{H}_\partial (D,g_t)\to \mathbb C^\ell$ is a linear isomorphism. In particular, the dimension of the space of boundary parallel harmonic vector fields equals the genus of $\partial D$. Equivalently, the equality $b_1(D)=\frac{1}{2}b_1 (\partial D)$ holds for domains in $\mathbb R^3$ (see \cite[Theorem 1]{CantarellaDeTurckGluck2002topology}). Here $b_1$ denotes the first Betti number, the dimension of the first (Co-) Homology group. For a general compact $3$-manifold $M$ with boundary one only has that $b_1 (M)\geq \frac{1}{2}b_1 (\partial M)$ \cite[Corollary 9.1.5]{Martelli2016AnIT}. It is easy to see that the inequality becomes strict after taking a connected sum with any closed manifold that is not a rational homology sphere.
\end{remark}

\begin{figure}[t]
\includegraphics[width=1\textwidth]{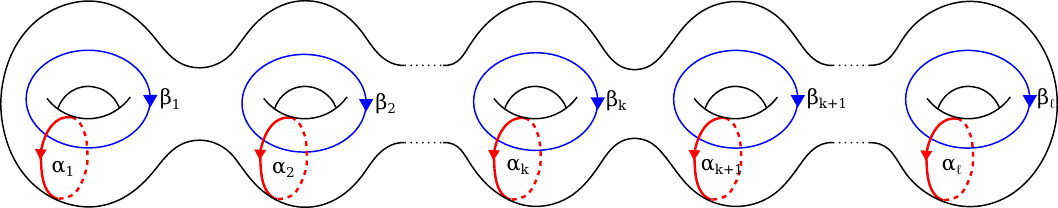}
\caption{The curves $\alpha_j$ and $\beta_j$. Note that the surface $\partial D$ need not be connected.}
\label{fig: alpha beta curve systems}
\centering
\end{figure}

A straightforward modification of the proof of \Cref{lem:analytic_resolvent} also yields continuous dependence of the resolvent (in the norm topology) on the metric with respect to the $C^1$-norm. We write $R_{g,L}(\lambda)$ for the resolvent of $\curl_{g,L}: H^1_L(D,g) \to L^2_\partial(D)$.

\begin{lemma}
    \label{lem:continuous_resolvent}
    Let $D \subseteq \mathbb R^3$ be a smoothly bounded domain with boundary of genus $\ell$. Let $g_0$ be a smooth metric on $D$, let $L \in \mathrm{LGr}(\mathbb C^{2\ell})$ and let $\lambda \in \mathbb C$ such that $R_{g_0,L}(\lambda)$ exists. Then there exist $\varepsilon > 0$ and $C > 0$ such that for all $g$ with $\|g-g_0\|_{C^1} \leq \varepsilon$, the resolvent $R_{g,L}(\lambda)$ exists and such that
    \begin{align}
        \| R_{g,L}(\lambda) - R_{g_0,L}(\lambda) \|_{L^2_\partial(D) \to H^1(\Omega^2(D))} \leq C \|g-g_0\|_{C^1}
    \end{align}
    hold for all such metrics $g$.
\end{lemma}

\begin{proof}
    Since $\star_g: H^1(\Omega^2(D)) \to H^1(\Omega^1(D))$ arises from a vector bundle isomorphism which, in local coordinates, has coefficients that are smooth functions of components of the metric and its inverse, it depends Lipschitz-continuously on $g$ in the following manner: Given a smooth metric $g_0$ on $D$, there exist $\varepsilon > 0$ and $C > 0$ such that
    \begin{align}
        \label{eq:c1dependence}
        \| \star_{g} - \star_{g_0} \|_{H^1(\Omega^2(D)) \to H^1(\Omega^1(D))} \leq C \| g - g_0 \|_{C^1} 
    \end{align}
    for all $g$ with $\| g - g_0 \|_{C^1} < \varepsilon$. Since $d: H^1(\Omega^2(D)) \to L^2(\Omega^2(D))$ and $\flux: H^1(\Omega^2(D)) \to \mathbb C^\ell$ are bounded and independent of $g$, it follows that
    \begin{align}
        \label{eq:c1dependence}
        \| \curl_g \oplus \flux - \curl_{g_0} \oplus \flux \|_{H^1(\Omega^2(D)) \to L^2(\Omega^2(D)) \oplus \mathbb C^\ell} \leq C \| g - g_0 \|_{C^1}.
    \end{align}
    Consider the inverse $A_{g}: L^2_{cl}(D) \to H^1_\partial(D)$ to the operator $\curl_g \oplus \flux$, as constructed in the proof of \Cref{lem:analytic_resolvent}. It is expressed by
    \begin{align*}
        A_g &:= A_{g_0}(\mathds{1} - (\curl_{g} \oplus \flux - \curl_{g_0} \oplus \flux) \circ A_{g_0})^{-1},
    \end{align*}
    whenever $\mathds{1} - (\curl_{g} \oplus \flux - \curl_{g_0} \oplus \flux) \circ A_{g_0}$ is invertible. By \eqref{eq:c1dependence} this holds for all $g$ with $\| g - g_0 \|_{C^1} < \varepsilon$ after possibly shrinking $\varepsilon$. It follows that there exists $C' > 0$ (depending on $C$ and $\|A_{g_0}\|_{L^2_{cl}(D)\oplus \mathbb C^\ell \to H^1_\partial(D)}$) such that
    \begin{align*}
        \| A_{g} - A_{g_0} \|_{L^2_{cl}(D)\oplus \mathbb C^\ell \to H^1_\partial(D)} \leq C' \| g - g_0 \|_{C^1}
    \end{align*}
    for all $g$ with $\| g - g_0 \|_{C^1} < \varepsilon$. Next, consider the Fredholm operator
    \begin{align*}
        R(1,g,L,\lambda) &= \mathds{1} - \big( \lambda \iota \oplus (\flux - B_L) \big) \circ A_g
    \end{align*}
    from the proof of \Cref{lem:analytic_resolvent}. It was shown there that $R(1,g,L,\lambda)$ is invertible if and only if $\lambda \not\in \sigma(\curl_{g,L})$, so $R(1,g_0,L,\lambda)$ is invertible.
    Observe that
    \begin{align*}
        &\|R(1,g,L,\lambda) - R(1,g_0,L,\lambda)\|_{L^2_{cl}(D)\oplus \mathbb C^\ell \to L^2_{cl}(D)\oplus \mathbb C^\ell} \\ = \ &\| \big( \lambda \iota \oplus (\flux - B_L) \big) \circ (A_g - A_{g_0})\|_{L^2_{cl}(D)\oplus \mathbb C^\ell \to L^2_{cl}(D)\oplus \mathbb C^\ell} \leq C'' \| g - g_0 \|_{C^1},
    \end{align*}
    for some $C'' > 0$. Since invertibility of Fredholm operators is an open condition, $R(1,g,L,\lambda)^{-1}$ is thus invertible for all $g$ with $\|g-g_0\|_{C^1} < \varepsilon$ after perhaps shrinking $\varepsilon > 0$ once more, and
    \begin{align*}
        \| A_g \circ R(1,g,L,\lambda)^{-1} - A_{g_0} \circ R(1,g_0,L,\lambda)^{-1} \|_{L^2_{cl}(D)\oplus \mathbb C^\ell \to H^1_\partial(D)} &\leq C''' \| g - g_0 \|_{C^1}.
    \end{align*}
    Since $R_{g,L}(\lambda)$ is the restriction of $\iota \circ A_g \circ R(1,g,L,\lambda)^{-1}$ to $L^2_{\partial}(D) \simeq L^2_{\partial}(D) \oplus \{0\}$, this establishes the claimed Lipschitz-continuous dependence of $R_{g,L}(\lambda)$ on $g$.
\end{proof}

Basic perturbation theory of operators with compact resolvent now shows that the spectrum of $\curl_{g,L}$ depends continuously on $g$ with respect to the $C^1$-norm. We record the proof for the convenience of the reader.

\begin{corollary}
    \label{cor:continuous_spectrum}
    Let $D$, $g_0$ and $L$ be as in \Cref{lem:continuous_resolvent}. Let $\lambda$ be an eigenvalue of $\curl_{g_0,L}$ of multiplicity $m \in \mathbb N$. Then there exist $\varepsilon > 0$, an open neighborhood $\mathcal U$ of $g_0$ in the space of smooth metrics on $D$ and $m$ continuous functions $\lambda_1,\ldots,\lambda_m: \mathcal U \to \mathbb R$ such that 
    \begin{align*}
        \sigma(\curl_{g,L}) \cap (\lambda-\varepsilon,\lambda+\varepsilon) = \{\lambda_1(g),\ldots,\lambda_m(g)\}
    \end{align*}
    for every $g \in \mathcal U$.
\end{corollary}

\begin{proof}
    We reduce the problem to a finite-dimensional one as follows. Let $\Gamma \subseteq \mathbb C$ be a closed curve encircling $\lambda$, but no other point of $\sigma(\curl_{g_0,L})$. Since $\Gamma$ is compact, there exists $\delta > 0$ such that $R_{g,L}(\zeta)$ exists for all $\zeta \in \Gamma$ and $g$ with $\|g-g_0\|_{C^1} < \delta$. Let $\mathcal U$ be the set of metrics with $\|g-g_0\|_{C^1} < \delta$. By \Cref{lem:continuous_resolvent} and the dominated convergence theorem, the spectral projection 
    \begin{align*}
        P_\Gamma(g) := \frac{1}{2\pi i} \int_\Gamma R_{g,L}(\zeta) d\zeta
    \end{align*}
    depends continuously on $g \in \mathcal U$. Since $P_\Gamma(g)$ is a continuously varying projection, its rank is constant, so $\mathrm{rk} P_\Gamma(g) = \mathrm{rk} P_\Gamma(g_0) = m$.
    
    Let $E_\Gamma(g) \subseteq L^2_\partial(D)$ denote the image of $P_\Gamma$, i.e.,\ the sum of the eigenspaces corresponding to eigenvalues encircled by $\Gamma$. The operator 
    $P_\Gamma(g_0) P_\Gamma(g): E_\Gamma(g_0) \to E_\Gamma(g_0)$ depends continuously on $g$ and acts as the identity when $g = g_0$. Hence one may shrink $\mathcal U$ so $(P_\Gamma(g_0) P_\Gamma(g))^{-1}$ exists on $\mathcal U$. Consider the family of operators on the finite-dimensional space $E_\Gamma(g_0)$ given by
    \begin{align*}
        A(g) := \left( P_\Gamma(g_0) P_\Gamma(g) \right)^{-1} \circ P_\Gamma(g_0) \circ \curl_{g} \circ P_\Gamma(g).
    \end{align*}
    Note that $A(g)$ is a concatenation of \emph{bounded} linear operators as follows:
    
    \begin{tikzcd}[column sep = huge]
        L^2_\partial(D) \arrow[r, "P_\Gamma(g)"] & H^1_\partial(D) \arrow[r, "\curl_g"] & L^2_\partial(D) \arrow[r, "P_\Gamma(g_0)"] & E_\Gamma(g_0) \arrow[r, "\left( P_\Gamma(g_0) P_\Gamma(g) \right)^{-1}"] & E_\Gamma(g_0)
    \end{tikzcd}
    
    \noindent All operators depend continuously on $g$. Plugging in the eigenbasis of $\curl_g$ shows that $A(g)$ shares its spectrum with $\curl_g$:
    \begin{align*}
        \sigma(A(g)) = \sigma(\curl_{g,L} \cap (\lambda-\varepsilon,\lambda+\varepsilon)
    \end{align*}
    Since $A(g)$ depends continuously on $g$ (in fact Lipschitz continuously with respect to the $C^1$ norm), the continuous dependence of the spectrum follows from the perturbation theory of finite-dimensional operators, see, e.g.,\ \cite[Theorem 5.1 and Theorem 5.2]{Kato} as well as the discussion leading up to \cite[Theorem 5.2]{Kato}.
\end{proof}

\begin{remark}\label{remark: continuity of spectrum}
    In the following, our metrics $g$ will be of the form $g = \Phi^\ast g_0$ for a fixed ambient metric $g_0$ and a diffeomorphism $\Phi$. \Cref{cor:continuous_spectrum} implies that the spectrum of $\curl_{\Phi^\ast g_0,L}$ depends continuously on $\Phi$ with respect to the $C^2$-norm. By conjugation, the spectrum of the curl operator on $\Phi(D)$ with domain $H^1_L(\Phi(D),g_0)$ depends continuously on $\Phi$ in the same way.
\end{remark}

As a consequence of \Cref{lem:analytic_resolvent} and (a generalization of) Rellich's theorem for analytic \emph{one-parameter} families of operators, we also obtain the following corollary. 

\begin{corollary}
    \label{cor:Rellich}
    Let $D$ and $L$ be as in \Cref{lem:continuous_resolvent}. Let $g_t$, $t \in (-\varepsilon,\varepsilon)$ be a real-analytic one-parameter family of smooth metrics on $D$. Consider an eigenvalue $\lambda$ of $\curl_{g_0,L}$ of multiplicity $m$ with corresponding eigenspace $E_\lambda$. Then there exist analytic functions $\lambda_1(t), \ldots, \lambda_m(t)$ and analytic one-parameter families $u_1(t), \ldots, u_m(t)$ in $L^2_\partial(D)$ such that
    \begin{enumerate}
        \item $u_j(t) \in H^1_L(D,g_t)$ and $\curl_{g_t,L} u_j(t) = \lambda_j(t)\, u_j(t)$ for all $j \in \{1,\ldots,m\}$,
        \item $(u_j(0))_{j=1}^m$ forms an orthonormal basis of $E_\lambda$, and
        \item $(u_j(t))_{j=1}^m$ is an orthonormal set with respect to $g_t$ for any $t \in (-\varepsilon,\varepsilon)$.
    \end{enumerate}
\end{corollary}

\begin{proof}
    Fix an arbitrary point $t_0 \in (-\varepsilon,\varepsilon)$. The first two paragraphs of the proof of \Cref{cor:continuous_spectrum} carry over to the analytic setting. After possibly shrinking the interval $(-\varepsilon,\varepsilon)$ to a smaller interval centered around $t_0$ to ensure invertibility of $\left(P_\Gamma(g_{t_0}) P_\Gamma(g_t) \right)^{-1}$, we thus obtain an analytic one-parameter family
    \begin{align*}
        A(g_t) := \left( P_\Gamma(g_{t_0}) P_\Gamma(g_t) \right)^{-1} \circ P_\Gamma(g_{t_0}) \circ \curl_{g_t} \circ P_\Gamma(g_t).
    \end{align*}
    of operators on the finite-dimensional space $E(g_{t_0})$. Note that $A(g_t)$ is symmetric with respect to the inner product $\left( \cdot, \cdot \right)_{t} := \left( P_\Gamma(g_t) \cdot, P_\Gamma(g_t) \cdot \right)_{g_t}$, but not necessarily with respect to a fixed inner product on $E(g_{t_0})$. Performing Gram--Schmidt orthogonalization on some fixed choice of real basis of $E(g_{t_0})$ with respect to $\left( \cdot, \cdot \right)_{t}$ yields a real-analytic family of maps $Q_t : \mathbb C^{m} \to E(g_{t_0})$ intertwining the standard inner product on $\mathbb C^m$ and $\left( \cdot, \cdot \right)_{t}$. Setting
    \begin{align*}
        B_t = Q_t^{-1} \circ A(g_t) \circ Q_t
    \end{align*}
    yields a real-analytic one-parameter family of symmetric $m\times m$-matrices. By Rellich's theorem for analytic one-parameter families of self-adjoint matrices \cite[Theorem 2]{RellichI}, there exist real-analytic functions $\lambda_1(t),\ldots,\lambda_m(t) \in \mathbb R$ and real-analytic families $v_1(t),\ldots,v_m(t) \in \mathbb C^m$ such that the following hold:
    \begin{enumerate}[label=(\arabic*)]
        \item $\sigma(B_t) = \{\lambda_1(t),\ldots,\lambda_m(t)\}$, counted with multiplicity,
        \item $B_t v_j(t) = \lambda_j(t) v_j(t)$ for all $j \in \{1,\ldots,m\}$,
        \item $v_j \cdot v_k = 1$ if $j = k$ and $0$ if $j \neq k$.
    \end{enumerate}
    Setting $u_j(t) = P_\Gamma(g_t) Q_t v_j(t)$ and tracing back through the definitions of $A(g_t)$ and $B_t$ completes the proof.
\end{proof}

This observation will let us use our Hadamard rule, which is proven in the following section, in the analysis of eigenvalues of higher multiplicity as well as of simple ones.

\section{The Hadamard rule}\label{section: The Hadamard rule}

In this section, we prove formula \eqref{eq:hadamard_rule_introduction} for the variation of a curl eigenvalue under domain deformations. In its statement, $(u\cdot v)$ denotes the (sesquilinear) pointwise inner product of two-forms, which is of course the same as the inner products $(\star u \cdot \star v)$ and $(\star u^\# \cdot \star v^\#)$ of the corresponding one-forms and their vector proxies, respectively.

\begin{theorem}[Hadamard rule]
\label{thm:hadamard_rule}
Let $D \subseteq \mathbb R^3$ be a smoothly bounded domain and $\Phi_t: D \to \mathbb R^3$ a smooth one-parameter family of embeddings of $D$. Write $D_t = \Phi_t (D)$, and $X_t = \frac{d\Phi_t}{dt}$. Furthermore, fix a Lagrangian $L \subseteq H^1_{dR}(\partial D;\mathbb C)$.
\begin{enumerate}[label=(\alph*)]
    \item Suppose that $u_t$ is a smooth $1$-parameter family of closed, boundary-parallel $2$-forms which solves $d ( \star u_t ) = \lambda_t u_t$ on $D_t$ and satisfies $[\Phi_t^\ast ( \star u_t)] \in L$ for all $t$.
    Then
    \begin{align}
        \Dot{\lambda} = - \lambda \int_{\partial D} (X \cdot \nu) \, |u|^2\, d\sigma,
    \end{align}
    where $\nu$ is the outward pointing unit normal of $\partial D$ and $d\sigma$ is the induced surface measure on $\partial D$.
    \item Let $v_t$ be another smooth $1$-parameter family of closed, boundary-parallel $2$-forms solving $d ( \star v_t ) = \mu_t v_t$ satisfying $[\Phi_t^\ast ( \star v_t)] \in L$ for all $t$. Suppose $\mu_0 = \lambda_0 \neq 0$ and $\left( u_t,v_t\right)_{L^2(D_t)} = 0$ for all $t$. Then 
    \begin{align}
        0 = \int_{\partial D} (X \cdot \nu) \, (u \cdot v) \, d\sigma.
    \end{align}
\end{enumerate}
\end{theorem}
\begin{proof}
We will first prove the identities \eqref{identity 1}--\eqref{identity 3} below, which follow from $L^2$-normalization, Lagrangian boundary conditions and parallel boundary conditions, respectively. In these, $w_t$ is a $1$-parameter family of eigenforms which may stand for both $u_t$ and $v_t$. We write $\rho_t$ as a placeholder for $\lambda_t$ and $\mu_t$, so that $d ( \star w_t ) = \rho_t w_t$. Then
\begin{align}\label{identity 1}
    \int_{\partial D} (X\cdot \nu)\,( \star u \cdot \star w)\,d\sigma + \int_{D} \star \Dot{u}\wedge \bar w +\int_{D} u\wedge \star \Dot{\bar w} =0
\end{align}
\begin{align}\label{identity 2}
    \int_{\partial D}\star\dot{u}\wedge\star \bar w= - \lambda \int_{\partial D} (X\iprod u) \wedge \star \bar w
\end{align}
\begin{align}\label{identity 3}
    \int_{\partial D} (X \iprod u) \wedge \star \bar w=\int_{\partial D} (X\cdot \nu) (u \cdot w)\,d\sigma
\end{align}

Identity (\ref{identity 1}) follows from differentiating the expression
\begin{align*}
\langle u_t,w_t\rangle_{D_t} = \int_{ D_t}u_t\wedge \star \bar w_t = \int_{D_t} (\star u_t \cdot \star w_t) \ dV_{\mathbb R^3},
\end{align*}
which is constantly $1$ if $w_t = u_t$ or constantly $0$ if $w_t = v_t$. Thus,
\begin{align*}
    0 = \frac{d}{dt}\lvert_{t=0} \int_{ D_t} u_t \wedge \star \bar w_t &= \int_{\partial D} (X \cdot \nu) (\star u \cdot \star w) + \int_{D} \dot u \wedge \star \bar w + \int_D u \wedge \star \dot{\bar w} \\
    &= \int_{\partial D} (X \cdot \nu) (\star u \cdot \star w) + \int_{D} \star \dot u \wedge \bar w + \int_D u \wedge \star \dot{\bar w},
\end{align*}
noting that $\int_D \dot u \wedge \star {\bar w} = \int_D (\star \dot u, \star w) \, dV_{\mathbb R^3} = \int_D \star \dot u \wedge \bar w$.

We begin the proof of \eqref{identity 2} by differentiating 
\begin{align*}
    \frac{d}{dt}\lvert_{t=0}\Phi_t^\ast (\star u_t)= \star \dot{u} + X\iprod (d\hodge u) + d(X\iprod \star u)
\end{align*}
Since $[\star w]$ and $[\Phi^\ast_t (\star u_t)]$ both belong to the same Lagrangian $L \subseteq H^1_{dR}(\partial D; \mathbb C)$,
\begin{align*}
    0 = \int_{\partial D} \Phi_t^\ast (\star u_t) \wedge \star \bar w.
\end{align*}
Differentiating this identity at $t=0$ and applying Stokes' theorem on $\partial D$ yields
\begin{align*}
    0 &= \int_{\partial D} \star \dot u \wedge \star \bar w+ \lambda \int_{\partial D} (X \iprod u) \wedge \star \bar w  + \int_{\partial D} d ( X \iprod \star u) \wedge \star \bar w  \\
    &= \int_{\partial D} \star \dot u \wedge \star \bar w + \lambda \int_{\partial D} (X \iprod u) \wedge \star \bar w  - \int_{\partial D} (X \iprod \star u) \wedge (d \hodge \bar w)\\
    &= \int_{\partial D} \star \dot u \wedge \star \bar w+ \lambda \int_{\partial D} (X \iprod u) \wedge \star \bar w,
\end{align*}
where the last term in the second line vanished since $\iota_{\partial D}^\ast(d \hodge \bar w) = \rho \, \iota_{\partial D}^\ast \bar w = 0$.

For identity (\ref{identity 3}), let  $X^\perp = (X\cdot \nu)\, \nu$ and $X^T = X-X^\perp$. The graded product rule for the interior product yields
\begin{align*}
    (X \iprod u) \wedge \star \bar w&= 
    X \iprod (u \wedge \star \bar w) - u \wedge (X \iprod \star \bar w) \\
    &= X^\perp \iprod (u \wedge \star \bar w) + X^T \iprod (u \wedge \star \bar w) - u \wedge (X \iprod \star \bar w)
\end{align*}
When pulled back to $\partial D$, the terms $X^T \iprod ( u \wedge \star \bar w)$ and $u \wedge (X \iprod \star \bar w)$ both vanish -- the former because $X^T$ is parallel to $\partial D$ and the latter because $\iota_{\partial D}^\ast u = 0$. Hence,
\begin{align*}
    \int_{\partial D} (X \iprod u) \wedge \star \bar w & = \int_{\partial D} X^\perp \iprod (u \wedge \star \bar w) \\
    &= \int_{\partial D} (X\cdot \nu)\, ( u \cdot w) \, (\nu \iprod dV_{\mathbb R^3}) \\
    & = \int_{\partial D} (X\cdot \nu)\, ( u \cdot w)\,d\sigma
\end{align*}

With these identities in hand we compute
\begin{align*}
    \Dot{\lambda}&= \frac{d}{dt}\lvert_{t=0} \int_{ D_t}d(\star u_t)\wedge \star \bar u_t \\
    &= \lambda \int_{\partial  D}\big(X\cdot \nu\big) \, \lvert \star u\rvert^2 \, d\sigma + \int_{ D} d(\star \Dot{u})\wedge \star \bar u + \lambda \int_{ D} u\wedge \star \Dot{\bar u} \\
    &=\lambda \int_{\partial D}\big(X \cdot \nu\big) \, \lvert \star u\rvert^2 \,d\sigma + \int_{\partial  D} \star \Dot{u}\wedge \star \bar u + \lambda \int_D \star \dot u \wedge \bar u + \lambda \int_D u\wedge \star \Dot{\bar u}\\
    &\overset{(\ref{identity 1})}{=} \int_{\partial  D} \star \Dot{u} \wedge \star \bar u \overset{(\ref{identity 2})}{=} - \lambda \int_{\partial  D} (X \iprod u) \wedge \star\bar u \\
    &\overset{(\ref{identity 3})}{=} - \lambda\int_{ \partial D} (X\cdot \nu) \lvert u\rvert^2\,d\sigma,
\end{align*}
yielding our Hadamard-type rule, part (a). The third line resulted from an application of Stokes' theorem to $\int_{\partial D} \star \dot u \wedge \star \bar u$, keeping in mind that $d \hodge \bar u = \lambda \bar u$ since $\lambda$ is real.

For part (b) we essentially repeat the same computations, keeping in mind that $\mu_0=\lambda_0$.
\begin{align*}
    0&= \frac{d}{dt}\lvert_{t=0} \int_{ D_t}d(\star u_t)\wedge \star \bar v_t \\
    &= 
    \lambda \int_{\partial  D} \big(X\cdot \nu\big) \, (\star u \cdot \star v) \, d\sigma + \int_{ D} d(\star \Dot{u})\wedge \star \bar v + \lambda \int_{ D} u\wedge \star \Dot{\bar v} \\
    &= \lambda \int_{\partial  D} \big(X\cdot \nu\big) \, (\star u \cdot \star v) \, d\sigma + \int_{\partial  D} \star \Dot{u}\wedge \star \bar v + \mu \int_D \star \dot u\wedge \bar v + \lambda \int_D u\wedge \star \Dot{\bar v} \\
    &\overset{(\ref{identity 1})}{=} \int_{\partial  D} \star \Dot{u} \wedge \star \bar v \overset{(\ref{identity 2})}{=} - \lambda \int_{\partial  D} (X \iprod u) \wedge \star \bar v \\
    &\overset{(\ref{identity 3})}{=} - \lambda \int_{ \partial D} (X\cdot \nu) \, ( u \cdot v)\,d\sigma
\end{align*}
Note that no assumptions on the connectedness of $\partial D$ were made.
\end{proof}

Theorem \ref{thm:hadamard_rule} provides a way of computing the variation of a curl eigenvalue in the direction of arbitrary deformation velocities. Given a domain $D_{\Phi_0}$ and $f \in  C^\infty(\partial D_{\Phi_0})$, we can construct an an analytic one-parameter family of diffeomorphisms $\Phi_t$ extending $\Phi_0$ such that $\left\langle \frac{d}{dt}|_{t=0} \Phi_t, \nu \right\rangle = f$. This can be achieved by extending $f\nu$ to a vector field $X$ on $D_{\Phi_0}$ and setting $\Phi_t = \Phi_0 + t X \circ {\Phi_0}$, which yields a diffeomorphism onto its image for all $t$ small enough. Since $\Phi_t^\ast g_{\mathbb R^3}$ is then a one-parameter family of metrics which is analytic in $t$, \Cref{cor:Rellich} then gives the existence of an analytic one-parameter family of curl eigenforms as required in the statement of Theorem \ref{thm:hadamard_rule}.

\section{Generic simplicity of the spectrum}\label{section: generic simplicity of the spectrum}

The following basic lemma implies that, for a residual set in the space of embeddings of a smoothly bounded domain $D$ into $\mathbb R^3$, the multiplicity of each associated eigenvalue of $\curl_L$ is locally constant. This will be useful in proving Theorem \ref{theorem:curl_multiplicity}.

\begin{lemma}
\label{lemma:abstract_multiplicity}
    Let $\mathcal X$ be a topological space, and $(\lambda_k)_{k\in \mathbb Z}$ a sequence of continuous functions such that $\lambda_k \leq \lambda_{k'}$ for any $k \leq k'$, and such that there does not exist a point where infinitely many $\lambda_k$ take the same value.
    The set of all $x \in \mathcal X$ with the property that, for each $k \in \mathbb Z$, there exist integers $k_1 \leq k \leq k_2$ and a neighborhood $U$ of $x$ where
    \begin{align*}
        \lambda_{k_1-1} < \lambda_{k_1} = \ldots = \lambda_k = \ldots = \lambda_{k_2} < \lambda_{k_2+1},
    \end{align*}
    is a countable intersection of dense open sets.
\end{lemma}

\begin{proof}
    For each pair $k_1 \leq k_2$, denote $\mathcal K_{k_1,k_2} = (\lambda_{k_1}-\lambda_{k_2})^{-1}\{0\}$. This set is closed, since $\lambda_{k_1}-\lambda_{k_2}$ is continuous. Fix $k \in \mathbb Z$ and consider
    \begin{align*}
        \mathcal U_k := \bigcup_{k_1 \leq k \leq k_2} \left( \mathcal K_{k_1-1,k_1}^c \cap \mathcal K_{k_1,k_2}^o \cap \mathcal K_{k_2,k_2+1}^c \right),
    \end{align*}
    where $o$ and $c$ denote the interior and complement of a set, respectively. As a union of open sets, $\mathcal{U}_k$ is open. To show that it is also dense, let $x \in \mathcal X$ be an arbitrary point. Let $k_1$ and $k_2$ be the smallest and largest index, respectively, such that there exists a neighborhood of $x$ on which $\lambda_{k_1} = \lambda_{k}$ and $\lambda_k = \lambda_{k_2}$. Then 
    \begin{align*}
        x \in \overline{\mathcal K_{k_1-1,k_1}^c} \cap \mathcal K_{k_1,k_2}^o \cap \overline{\mathcal K_{k_2,k_2+1}^c} \subseteq \overline{\mathcal U_k},
    \end{align*}
    which shows $\mathcal U_k$ is dense, as claimed. The set of interest is the (countable) intersection of all $\mathcal U_k$, $k \in \mathbb Z$.
\end{proof}

We split the proof of \Cref{thm:simple_eigenvalues} into two parts. First, we use the Hadamard rule to derive strong consequences from the assumption that the spectrum of $\curl$ is not generically simple. We then show these consequences lead to a contradiction, hence proving \Cref{thm:simple_eigenvalues}. 

Recall that in the formulation of \Cref{thm:simple_eigenvalues}, we fixed a domain $D_0 \subseteq \mathbb R^3$ and a Lagrangian $L \subseteq H^1_{dR}(\partial D_0; \mathbb R)$. We then considered the Frech\'et manifold $\mathcal X(D_0)$ of smooth embeddings $\Phi: \overline{D_0} \to \mathbb R^3$, and associated to each $\Phi \in \mathcal X$ its image $D_\Phi := \Phi(D_0)$ and the Lagrangian $L_\Phi := (\Phi^{-1})^\ast L$. To formulate our intermediate result, we denote the nonzero eigenvalues of the curl operator $\curl_{L_{\Phi}}: H^1_{L_\Phi}(D_\Phi) \to L^2_\partial(D_\Phi)$, counted with multiplicity, by $\lambda^L_k(\Phi)$, where $k \in \mathbb Z\setminus\{0\}$.

\begin{proposition}
\label{theorem:curl_multiplicity}
Let $D_0 \subseteq \mathbb R^3$ and $L \subseteq H^1_{dR}(\partial D_0; \mathbb R)$ be as in \Cref{thm:simple_eigenvalues}. Then for a comeagre subset of diffeomorphisms $\Phi \in \mathcal X(D_0)$, the following holds:
\begin{enumerate}
    \item For each $k \in \mathbb Z \setminus\{0\}$, the multiplicity of $\lambda_k^L(\Phi')$ is constant for all $\Phi'$ in a neighborhood $\mathcal U_k \subseteq \mathcal X(D_0)$ of $\Phi$, and equals either one or two.
    \item In the latter case, the components of any orthonormal basis $(u_1,u_2) \in H^1_{\partial\partial}(D_{\Phi'})$ of the eigenspace corresponding to $\lambda_k^L(\Phi')$ are pointwise orthogonal and of equal length on $\partial D_{\Phi'}$ for any $\Phi' \in \mathcal U_k$. Furthermore, $\star u_1$ and $\star u_2$ pull back to harmonic forms on $\partial D_{\Phi'}$ with its induced metric.
\end{enumerate}
\end{proposition}

\begin{proof}
    Fix any $\Phi_0 \in \mathcal X(D_0)$ and choose a real number $\mu$ not in the spectrum of $\curl_{L_{\Phi_0}}$ on $D_{\Phi_0}$. Since the spectrum of $\curl_{L_\Phi}$ depends continuously on $\Phi$ by \Cref{cor:continuous_spectrum}, we may choose a neighborhood $\mathcal U \subseteq \mathcal X(D_0)$ of $\Phi_0$ small enough such that $\mu$ does not lie in the spectrum of $\curl_{L_{\Phi}}$ on $D_\Phi$ for all $\Phi \in \mathcal U$. This allows us to define the family of operators
    \begin{align*}
        R_\Phi&: L^2_\partial(D_0) \to L^2_\partial(D_0), \\
        R_\Phi \, u &= \Phi^\ast \circ (\mu-\curl_{L_\Phi})^{-1} \circ (\Phi^{-1})^\ast \, u.
    \end{align*}
    The operator $R_\Phi$ coincides with the resolvent of the curl operator 
    \begin{align*}
        d \circ \star_{\Phi^\ast g} : H^1_{L}(D_0,\Phi^\ast g) \to L^2_\partial(D_0).
    \end{align*}
    It is self-adjoint with respect to the inner product induced by $\Phi^\ast g$.

    Since the eigenvalues of $R_\Phi$, and thus also the eigenvalues of $\curl_{L_\Phi}$ on $D_\Phi$, depend continuously on $\Phi$ by \Cref{cor:continuous_spectrum}, we may apply \Cref{lemma:abstract_multiplicity} to conclude that there exists a comeagre subset of $\mathcal U$ where the multiplicity of each eigenvalue $\lambda_k^L(\Phi)$ is locally constant.

    Let $\Phi_1 \in \mathcal U$ be one such map on which the multiplicity of every eigenvalue $\lambda_k^L(\Phi)$ is locally constant. Fix $k \in \mathbb Z \setminus \{0\}$ such that $\lambda_{k-1}^L(\Phi_1) < \lambda_{k}^L(\Phi)$ and let $\mathcal U_1$ be a neighborhood of $\Phi_1$ on which
    \begin{align*}
        \lambda_{k-1}^L(\Phi) < \lambda_{k}^L(\Phi) = \ldots = \lambda_{k+m}^L(\Phi) < \lambda_{k+m+1}^L(\Phi)
    \end{align*}
    for some $m \in \mathbb N$ and all $\Phi \in \mathcal U_1$.
    Consider any eigenfunction $u \in L^2_\partial(D_0)$ of $R_{\Phi_1}$ corresponding to the curl eigenvalue $\lambda = \lambda_k^L(\Phi_1)$. Choose a small circle $C \subseteq \mathbb C$ encircling $(\mu - \lambda)^{-1}$, but no other eigenvalue of $R_{\Phi_1}$. After possibly shrinking $\mathcal U_1$, the Riesz projection
    \begin{align*}
        P_\Phi = \frac{1}{2\pi i} \int_{C} (\zeta - R_\Phi)^{-1}
    \end{align*}
    is defined and depends continuously on $\Phi \in \mathcal U_1$, see the proof of \Cref{cor:continuous_spectrum}. Since the multiplicity of $\lambda_k^L$ is assumed to be constant on $\mathcal U_1$, the range of $P_\Phi$ consists of eigenfunctions of $R_\Phi$ for all $\Phi \in \mathcal U_1$. We may construct a family of eigenfunctions $u_\Phi \in L^2_\partial(D_0)$ satisfying
    \begin{align*}
        R_{\Phi} u_\Phi &= (\mu - \lambda_k^L(\Phi))^{-1} u_\Phi, & \|u\|_{L^2(D_0,\Phi^\ast g)} &= 1, & u_{\Phi_1} &= u,
    \end{align*}
    by simply setting $u_\Phi = \|P_\Phi u\|_{L^2(D_0,\Phi^\ast g)}^{-1} P_\Phi u$, after possibly shrinking $\mathcal U_1$ to ensure $P_\Phi u \neq 0$.

    Write $D := D_{\Phi_1}$. Given $f \in  C^\infty(\partial D)$, we consider an analytic $1$-parameter family of diffeomorphisms $\Phi(t)$ in $\mathcal U_1$ with $\Phi(0) = \Phi_1$ and $\left\langle \frac{d}{dt}|_{t=0} \Phi, \nu \right\rangle = f$, giving rise to the analytic $1$-parameter family of metrics $\Phi(t)^\ast g_{\mathbb R^3}$ on $D_0$. \Cref{lem:analytic_resolvent} then implies that $P_{\Phi(t)}$ depends analytically on $t$, and \Cref{thm:hadamard_rule} yields
    \begin{align*}
        \frac{d}{dt}\vert_{t=0} \, \lambda_k^L(\Phi(t)) = - \lambda \int_{\partial D} f|u|^2\,d\sigma.
    \end{align*}
    For simplicity, denote $\frac{d}{dt}\vert_{t=0}\lambda_k^L(\Phi(t))$ by $\dot\lambda_f$. By polarization, we find that
    \begin{align*}
        \int_{\partial D} f\left( u_j \cdot u_k \right) \,d\sigma = - \frac{\dot\lambda_f}{\lambda} \, \delta_{j,k}
    \end{align*}
    for any real valued orthonormal basis $(u_j)_{j=1}^m$ of $E_\lambda$ (note that such a basis always exists because $\curl_L$ commutes with complex conjugation). Indeed, 
    \begin{align*}
        \int_{\partial D} f \left( u_j \cdot u_k \right) \,d\sigma &= \frac12 \left( \int_{\partial D} f\left|\tfrac{u_1+u_2}{\sqrt2}\right|^2\,d\sigma + \int_{\partial D} f\left|\tfrac{u_1-u_2}{\sqrt2}\right|^2\,d\sigma - \int_{\partial D} f|u_1|^2\,d\sigma - \int_{\partial D} f|u_2|^2\,d\sigma \right) \\
        &=  - \frac12 \left( \frac{\dot\lambda_f}{\lambda} + \frac{\dot\lambda_f}{\lambda} - \frac{\dot\lambda_f}{\lambda} - \frac{\dot\lambda_f}{\lambda} \right) = 0
    \end{align*}
and
    \begin{align*}
        \int_{\partial D} f|u_1|^2\,d\sigma &= -\frac{\dot\lambda_f}{\lambda} = \int_{\partial D} f|u_2|^2\,d\sigma.
    \end{align*}
    Since these identities hold for any $f\in C^\infty (\partial D)$, the matrix $\left( u_i \cdot u_j \right)_{i,j = 1}^m$ is a multiple of the identity at every point $p \in \partial D$. This conclusion in fact holds for any $\Phi \in \mathcal U_1$, since the preceding argument only used local constancy of the eigenvalue multiplicity and can therefore be repeated at any $\Phi \in \mathcal U_1$.

    There does not exist a curl eigenform $u$ which vanishes identically on $\partial D$ (cf.\ \cite[proof of Proposition 2.1]{EncisoPeralta2020}). Consequently, there exists an open subset of $\partial D$ on which all $u_i$ are nonvanishing and pointwise orthogonal to one another. As the $u_i$ are boundary parallel, it follows that $m \leq 2$. The eigenspace $E_\lambda$ is thus either one-dimensional or spanned by two eigenforms $u_1$ and $u_2$ which are orthogonal and of equal length at each point $p \in \partial D$. 

    The assumption that the Lagrangian $L$ is \emph{real}, which was not used until now, helps to deal with the latter case. If $L$ is real, the operator $\curl_L$ commutes with complex conjugation, hence all eigenforms in this proof can be taken to be real valued. This allows us to show $\star u_1$ and $\star u_2$ are in fact harmonic with respect to the induced metric on $\partial D$. Write $\alpha = \iota_{\partial D}^\ast \hodge u_1$, $\beta = \iota_{\partial D}^\ast \hodge u_2$, and let $\star'$ denote the Hodge-star on $\partial D$ with its induced metric. Consider a connected open set $V \subseteq \partial D$ where $|\alpha| \neq 0$. Since $\beta$ and $\alpha$ are orthogonal and of equal length, $\beta = \pm \star' \alpha$ on $V$, with the sign constant throughout $V$ by continuity. Without loss of generality, assume that $\beta = \star' \alpha$. Part of the boundary condition $d \star u_1 \in L^2_{\partial}(D)$ is that $\iota_{\partial D}^\ast(d \star u_1) = 0$, and hence $d\alpha = 0$. Analogously, $d\beta = 0$. Together, we have $d \beta = d \star' \! \beta = 0$, i.e.\ $\beta$ is harmonic on $V$ with its induced metric, and analogously $\alpha$ as well. By continuity, $\alpha$ and $\beta$ are harmonic on $\overline V$. Repeating this argument on all connected components of the set $\{p \in \partial D: |u_1| \neq 0\}$ shows that $\alpha$ and $\beta$ are harmonic on its closure. In the interior of its complement, $\alpha$ and $\beta$ vanish identically. Hence, both are in fact harmonic on all of $\partial D$. It now follows that on each connected component $M$ of $\partial D$, either $\beta = \star' \alpha$ or $\beta = -\star' \alpha$ holds. This is because at least one of the harmonic forms $\beta - \star'\alpha$ or $\beta + \star'\alpha$ vanishes on an open subset of $M$ and therefore on all of $M$ by the unique continuation theorem for harmonic differential forms \cite[Theorem 1]{Aronszajn1962}.
\end{proof}

We thank Daniel Peralta-Salas for the important observation that boundary parallel curl eigenforms which are orthogonal and of equal length on the boundary must restrict to harmonic forms on the boundary with respect to its induced metric. The proof of Theorem \ref{thm:simple_eigenvalues} in the case of disconnected boundary rests on this observation.

\begin{proof}[Proof of \Cref{thm:simple_eigenvalues}]
According to Proposition \ref{theorem:curl_multiplicity}, we are presented with a dichotomy: Either the spectrum of $\curl_L$ is simple for generic domains diffeomorphic to $D_0$ or there exists an integer $k$, an embedding $\Phi_1:D_0\to \mathbb R^3$, an open subset $\mathcal{U}_k$ in the space of smooth maps from $D_0$ to $\mathbb R^3$ containing $\Phi_1$ and a family of \emph{real-valued} orthonormal bases $u_1 (\Phi)$ and $u_2 (\Phi)$ of the eigenspace associated to $\lambda_k^L(\Phi)$ such that the $1$-forms $\star u_i (\Phi)$ restrict to harmonic forms on $\partial D_\Phi$ with respect to the induced metric for all $\Phi \in \mathcal{U}_k$ and satisfy $\star u_2=\star^\prime (\star u_1)$, where $\star^\prime$ is the Hodge-star on $\partial D$. We will show that the second alternative is impossible. For this we distinguish two cases, depending on whether $\partial D_0$ is connected.

    \textbf{Case 1: $\partial D$ is connected.} Here, we use the self-adjointness of $\curl_L$ to obtain
        \begin{align*}
            0 &= \left| \left\langle\curl u_1, u_2 \right\rangle - \left\langle u_1, \curl  u_2 \right\rangle \right| \\
            &= \left| \int_{\partial D} \star u_1 \wedge \star u_2 \right| \\
            &= \left| \int_{\partial D} \star u_1 \wedge \star' (\star u_1) \right|
           \\
           &= \int_{\partial D} |\hodge u_1|^2 d\sigma,
        \end{align*}
        which implies that $u_1$ vanishes identically on $\partial D$. As already mentioned, this possibility is ruled out by an argument presented in \cite[Proof of Proposition 2.3]{EncisoPeralta2020}. If $\partial D$ is connected, it is therefore impossible for two real valued curl eigenforms $u_1$ and $u_2$ of the same eigenvalue to be orthogonal and of equal length on the boundary.
        
        \textbf{Case 2: $\partial D$ is not connected.} Here, the argument above fails, since it is possible that $\star u_1 = \star' (\star u_2)$ on one boundary component and $\star u_1 = -\!\star'\! (\star u_2)$ on another. Instead, we will make an argument that plays the finite-dimensionality of the space of harmonic forms on a fixed boundary component of $\partial D$ against the infinite-dimensionality of the space of perturbations of the remaining boundary components via the Cauchy--Kovalevskaya theorem for curl eigenforms \cite[Theorem 3.1]{EncisoPeralta2012}. We will make this precise in the next paragraph.
        
        Fix a boundary component $M \subseteq \partial D$ and let $(\Phi_t)_{t \in \mathbb [-1,1]^{k}}$ be a real-analytic $m$-parameter family of diffeomorphisms in $\mathcal U_1$ ($m$ to be fixed later) such that $M \subseteq \partial D_{\Phi_t}$ for all $t \in [-1,1]^m$ and $\partial D_{\Phi_{t_1}} \cap \partial D_{\Phi_{t_2}}$ contains no open subset of $\partial D_{\Phi_{t_1}}$ but $M$ if $t_1 \neq t_2$. Since the multiplicity of $\lambda(t) := \lambda_k(D_{\Phi_t})$ is constantly equal to two, by the same argument as in the proof of \Cref{theorem:curl_multiplicity} we obtain analytic $m$-parameter families of real-valued eigenforms $u_1(t)$ and $u_2(t)$ such that
        \begin{align*}
            &\curl u_j(t) = \lambda(t) u_j(t), \\
            &\left\langle u_i(t),u_j (t)\right\rangle_{L^2(D_t)} = \delta_{i,j}
        \end{align*}
        for $i,j \in \{1,2\}$. Furthermore, by \Cref{theorem:curl_multiplicity}, $\iota_M^\ast (\star u_1(t))$ and $\iota_M^\ast (\star u_2(t))$ are pointwise orthogonal harmonic $1$-forms on $M$ for all $t \in [-1,1]^m$. Let $\ell'$ denote the genus of $M$. The map $F:[-1,1]^m \to \mathbb R^{1+2\ell'}$ given by
        \begin{align*}
            F(t) = \left(\lambda(t),[\iota_M^\ast (\star u_1(t))]_{H^1_{dR}(M)}\right)
        \end{align*}
        is real analytic. If $m = 3 + 2\ell'$, by Sard's theorem, a generic level set of $F$ is a submanifold of codimension $1 + 2\ell'$, and hence two-dimensional. Let us henceforth restrict $t$ to such a level set $\mathcal S$. Then $[\iota_M^\ast (\star u_1(t))]_{H^1_{dR}(M)}$ is constant, and since $\iota_M^\ast (\star u_1(t))$ is harmonic, this means $\iota_M^\ast (\star u_1(t))$ is constant. As $\lambda(t)$ is constant as well, the uniqueness part of the Cauchy-Kovalevskaya theorem for curl eigenforms obtained in \cite[Theorem 3.1]{EncisoPeralta2012} implies that $\star u_1(t_1) = \star u_1(t_2)$ on $D_{t_1} \cap D_{t_2}$ for all $t_1, t_2 \in \mathcal S$. On the domain $D_{\mathcal S} = \bigcup_{t \in \mathcal S} D_t$, we can thus find $u_1^{\mathcal S}$ such that $\star u_1^{\mathcal S} |_{x} = \star u_1(t) |_{x}$ whenever $x \in D_t$. Define $u_2^{\mathcal S}$ analogously. The parallel boundary conditions imply that the vector proxies $(\star u_1^{\mathcal S})^\sharp$ and $(\star u_2^{\mathcal S})^\sharp$ are tangential to $\partial D_t$ for all $t \in \mathcal S$. Equivalently, $\partial D_t$ integrates the (possibly singular) plane field $\ker \star(\star u_1^{\mathcal S} \wedge \star u_2^{\mathcal S})$. As a consequence, $\partial D_{t_1} \cap \partial D_{t_2} \setminus M$ contains an open subset of $\partial D_{t_1}$ if $\partial D_{t_1} \cap \partial D_{t_2} \setminus M$ contains a single point $p \in D_{\mathcal S}$ where $u_1^{\mathcal S}|_{p} \neq 0$. The vanishing set of $u_1^{\mathcal S}$ is at most two-dimensional and so cannot contain $\partial D_t$ for all $t \in \mathcal S$. Thus, there exists an open set $O \subseteq D_{\mathcal S}$ where $u_1^{\mathcal S}$ does not vanish, and a nonempty open set $\mathcal S' \subseteq \mathcal S$ such that $\partial D_t$ intersects $O$ nontrivially for all $t \in \mathcal S'$. Since $\partial D_t$ and $\mathcal S'$ are two-dimensional, and $O$ is only three-dimensional, there must exist $t_1,t_2 \in \mathcal S'$ such that $\partial D_{t_1} \cap \partial D_{t_2} \cap O \neq \emptyset$. Then, $\partial D_{t_1} \cap \partial D_{t_2}$ contains a nontrivial open subset of $\partial D_1$ other than $M$, contradicting our assumption on the family $D_t$, and hence finishing our argument.
\end{proof}

\begin{remark}
    The proof of \Cref{thm:simple_eigenvalues} in the case of connected boundary applies verbatim to any family of domains large enough so that the space of ``deformation velocities'' $\left\langle X,\nu \right\rangle$ spans a dense subspace of $L^2(\partial D)$. If $\partial D$ is disconnected, we additionally need to be able to to perturb $\partial D$ while keeping one component fixed. Both of these properties are satisfied, for example, if one restricts attention to domains with analytic boundary.
\end{remark}

\section{Local extrema of the eigenvalue functionals}\label{section: critical points of the eigenvalue functionals}

This section is concerned with proving \Cref{thm:extremal_eigenvalues}, i.e.\ with the characterization of local extrema for the $k^{th}$ eigenvalue functionals $\lambda_k^L$ over the space $\mathcal X_1(D_0)$ of domains diffeomorphic to $D_0$ with volume $1$.

Local extrema of $\lambda_k^L$ often arise when an eigenvalue cluster has the property that along any volume-preserving perturbation, the lowest eigenvalue splitting off the cluster is non-increasing, while the highest eigenvalue is non-decreasing. This picture provides the intuition behind the first part of \Cref{thm:extremal_eigenvalues}, which parallels a theorem obtained by El Soufi and Ilias \cite{ElSoufiIllias2008} for the Laplace-Beltrami operator acting on functions on a closed Riemannian manifold, and is proven using the same general strategy.

\begin{proof}[Proof of \Cref{thm:extremal_eigenvalues}]
    Suppose $\Phi_0$ is a local extremum of $\lambda_k^L$ in $\mathcal X_1(D_0)$. Equivalently, $\Phi$ is a local extremum for dimensionless functional $|D_\Phi|^{\frac13} \lambda^L_k(\Phi)$ over the entire space $\mathcal X(D_0)$. For the rest of the proof, we write $D = D_{\Phi_0}$ and recenter $\mathcal X(D_0)$ at $D$ via the isomorphism $\mathcal X(D_0) \to \mathcal X(D)$, $\Phi \mapsto \Phi \circ \Phi_0^{-1}$.
    
    Assume that $\lambda := \lambda_k^L(\mathrm{id})$ has multiplicity $m$. 
    Consider any function $f \in C^\infty(\partial D)$. As in the proof of \Cref{thm:simple_eigenvalues}, let $\Phi_t$ be an analytic $1$-parameter family of diffeomorphisms defined on a neighborhood of $D$, such that $\Phi_0 = \mathrm{id}$ and $\left\langle X, \nu \right\rangle = f$, where $X$ denotes the vector field $\frac{d}{dt}\rvert_{t=0} \Phi_t$ and $\nu$ the outward unit normal to $\partial D$. \Cref{cor:Rellich} yields $\Lambda_1(t),\ldots,\Lambda_m(t)$ and $u_1(t),\ldots,u_m(t) \in H^1_{L_{\Phi_t}}(D_{\Phi_t})$ such that $(u_j(0))_{j=1}^m$ is an orthonormal basis of the eigenspace $E_\lambda$ and
    \begin{align*}
        \curl u_j(t) = \Lambda_j(t) u_j(t).
    \end{align*}
    \Cref{thm:hadamard_rule} then implies that 
    \begin{align*}
        \Lambda_j'(0) = - \lambda \int_{\partial D} f |u_j|^2 \text{ and } 0 = \int_{\partial D} f (u_j \cdot u_k).
    \end{align*}
    Suppose first that the (finite dimensional) space spanned by $\{|u|^2|_{\partial D}: u \in E_\lambda\}$ does not contain the constant function $1$. Then there exists a function $f \in C^\infty(\partial D)$ with $\int_{\partial D} f |u|^2 d\sigma = 0$ for all $u \in E_\lambda$ but $\int_{\partial D} f \, d\sigma = 1$. Construct an analytic $1$-parameter family $\Phi_t$ as above which satisfies $\left\langle X, \nu \right\rangle = f$. Then there exist $u_1,\ldots,u_m \in E_\lambda$ with
    \begin{align}
        \label{eq:derivative_of_normalized_eigenvalue}
        \begin{split}
        \frac{d}{dt}\big\rvert_{t=0} \, \left( |D_t|^{\frac{1}{3}} \Lambda_j(t) \right) &= \Lambda_j'(0)|D_t|^{\frac{1}{3}} + \frac13 \Lambda_j(0)|D_t|^{-\frac{2}{3}} \frac{d}{dt} |D_t| \\
        &= - \lambda |D|^{\frac13} \int_{\partial D} f |u_j|^2 d\sigma + \frac13 \lambda |D|^{-\frac{2}{3}} \int_{\partial D} f \, d\sigma    
        \end{split}
    \end{align}
    for all $j \in \{1,\ldots,m\}$.
    Thus, all normalized eigenvalues $|D_t|^{\frac{1}{3}} \Lambda_j(t)$ are greater in absolute value than $\lambda$ for some $t > 0$, and smaller in absolute value for some $t < 0$, contradicting the assumption that $D$ was a local extremum for $\lambda_k^L$.

    If $\mathrm{span}_{\mathbb R}(\{|u|^2|_{\partial D}: u \in E_\lambda\})$ contains $1$ but the convex cone spanned by $\{|u|^2|_{\partial D}: u \in E_\lambda\}$ does not, we may choose $f \in \mathrm{span}_{\mathbb R}(\{|u|^2|_{\partial D}: u \in E_\lambda\})$ with $\int_{\partial D} f |u|^2 d\sigma \leq 0$ for all $u \in E_\lambda$, but $\int_{\partial D} f \, d\sigma > 0$, and arrive at the same contradiction as in the previous paragraph.

    Hence, $1$ lies in the cone spanned by $\{|u|^2|_{\partial D}: u \in E_\lambda\}$. An elementary linear algebra calculation, which we isolate in \Cref{lem:abstract_cone} below, shows that there exists an orthonormal basis $(u_j)_{j=1}^m$ of $E_\lambda$ and $\tau_1,\ldots,\tau_m \geq 0$ such that $|\tau_1 u_1|^2 + \ldots + |\tau_m u_m|^2 = 1$ on $\partial D$. After discarding $\tau_j u_j$ whenever $\tau_j = 0$, we have found an orthogonal set of eigenforms with the claimed property.

    It remains to consider the special case of the lowest positive eigenvalue. (The case of the highest negative eigenvalue is analogous). Here, we must have $\frac{d}{dt}\big\rvert_{t=0} \, \left( |D_t|^{\frac{1}{3}} \Lambda_j(t) \right) = 0$ for all $j \in \{1,\ldots,m\}$ along any $1$-paramater family $\phi_t$, since otherwise we could further decrease the lowest eigenvalue of the eigenvalue cluster $\Lambda_1,\ldots,\Lambda_m$, which by definition decreases $\lambda_1^L$. Hence,
    \begin{align*}
        \int_{\partial D} f|u_j|^2 d\sigma = \frac13 |D|^{-1} \int_{\partial D} f \, d\sigma 
    \end{align*}
    for all $f \in C^\infty(\partial D)$. Part (b) of \Cref{thm:hadamard_rule} furthermore asserts that
    \begin{align*}
        \int_{\partial D} f(u_j \cdot u_k) d\sigma = 0
    \end{align*}
    for all $j \neq k$. It follows that
    \begin{align*}
        \int_{\partial D} f|u|^2 d\sigma = \frac13 |D|^{-1} \int_{\partial D} f \, d\sigma 
    \end{align*}
    for any $L^2$-normalized $u \in E_\lambda$. Since $f \in C^\infty(\partial D)$ was arbitrary, $|u|^2 = \frac{1}{3} |D|^{-1}$ on $\partial D$.

    Having established this, one may proceed as in the proof of \cite[Proposition 3.6]{Gerner2023} to show that $\dim E_\lambda = 1$. Let us note that the key idea there is to show by direct calculation that $\dim E_\lambda \geq 2$ would imply the induced metric on $\partial D$ is flat, which is not possible for a smoothly bounded domain $D \subseteq \mathbb R^3$.
\end{proof}

\begin{lemma}
    \label{lem:abstract_cone}
    Let $B: V \times V \to W$ be a symmetric bilinear map between a complex, finite dimensional inner product space $V$ and another complex vector space $W$.
    For any element $c$ in the cone spanned by $\left\{B(v,\bar v), v \in \mathbb R^{m} \right\}$, there exists an orthonormal basis $(v_j)_{j=1}^m$ of $V$ and coefficients $\tau_1,\ldots,\tau_m \geq 0$ such that $c = \sum_{j=1}^m B(\tau_j v_j,\tau_j \bar v_j)$.
\end{lemma}

\begin{proof}
    Since $c \in \mathcal C$, there exist $\eta_1,\ldots,\eta_N \geq 0$ and $u_1,\ldots,u_N \in V$ such that
    \begin{align*}
        c = \sum_{k=1}^N \eta_k B(u_k,\bar u_k).
    \end{align*}
    Pick an orthonormal basis $(w_j)_{j=1}^m$ of $V$. Write $u_k = \sum_{j=1}^m \nu_k^j w_j$. Then 
    \begin{align*}
        c = \sum_{k=1}^N \eta_k \sum_{i,j=1}^m \nu_k^i\bar \nu_k^j B(w_i,\bar w_j) = \sum_{i,j=1}^m \left( \sum_{k=1}^N \eta_k \nu_k^i \bar \nu_k^j \right) B(w_i,\bar w_j).
    \end{align*}
    Denote $\alpha_{ij} = \eta_k \nu_k^i\bar \nu_k^j$. Then $\alpha$ is a hermitian, positive semidefinite $m \times m$ matrix. Thus, there exists a unitary matrix $\beta$ and a diagonal matrix $\tau = \mathrm{diag}(\tau_1,\ldots,\tau_m)$ with nonnegative entries such that $\beta \tau \beta^\ast = \alpha$. In other words,
    \begin{align*}
        \sum_{k=1}^m \tau_k \beta_{i,k} \overline{\beta_{j,k}} = \alpha_{i,j}.
    \end{align*}
    Setting $v_k = \sum_{j=1}^m \beta_{j,k} w_j$ for $k = 1,\ldots,m$ yields
    \begin{align*}
        &\sum_{k=1}^m \tau_k B(v_k,\bar v_k) = \sum_{k=1}^m \tau_k \sum_{i,j = 1}^m \beta_{i,k} \overline{\beta_{j,k}} B(w_i,\bar w_j) \\ = \ & \sum_{i,j = 1}^m \left( \sum_{k=1}^m \tau_k \beta_{i,k} \overline{\beta_{j,k}} \right) B(w_i,\bar w_j) = 
        \sum_{i,j = 1}^m \alpha_{i,j} B(w_i,\bar w_j) = c,
    \end{align*}
    and thus $c = \sum_{k=1}^m B(\sqrt{\tau_k} v_k,\sqrt{\tau_k} \bar v_k)$. Normalizing $(v_k)_{k=1}^m$ yields the result.
\end{proof}

\newpage

\bibliographystyle{amsplain0}
\bibliography{main.bib}

\end{document}